\documentclass{article}
\usepackage[utf8]{inputenc}
\usepackage{fancyhdr}
\usepackage{amsmath,amssymb,amsthm}
\usepackage{tikz}
\usepackage{wrapfig} %
\usepackage{graphicx} %
\usepackage{tikz-cd}
\usepackage[margin=1.5in]{geometry} %
\usepackage[backend=bibtex,style=alphabetic]{biblatex}
\usepackage{tipa} %
\usepackage{mathtools} %
\usepackage{caption}
\usepackage{thmtools, thm-restate} %

\usetikzlibrary{decorations.pathreplacing} %

\graphicspath{graphics/}

\DeclareMathOperator{\Ca}{\mathcal{C}}

\DeclareMathOperator{\La}{\mathcal{L}}

\DeclareMathOperator{\Ra}{\mathcal{R}}

\DeclareMathOperator{\Cb}{\mathbb{C}}

\DeclareMathOperator{\Hb}{\mathbb{H}}
\DeclareMathOperator{\Kb}{\mathbb{K}}

\DeclareMathOperator{\Nb}{\mathbb{N}}

\DeclareMathOperator{\Rb}{\mathbb{R}}
\DeclareMathOperator{\Sb}{\mathbb{S}}

\DeclareMathOperator{\Zb}{\mathbb{Z}}

\newcommand{\modal}{{\Diamond}}
\newcommand{\modaltwo}{{ \blacklozenge }}
\newcommand{\comodal}{{ \Box }}

\DeclareMathOperator{\inv}{^{-1}}
\DeclareMathOperator{\id}{\mathsf{id}}

\DeclareMathOperator{\Aut}{\mathsf{Aut}}
\DeclareMathOperator{\BAut}{\mathsf{BAut}}

\DeclareMathOperator{\Prop}{\textbf{Prop}}
\DeclareMathOperator{\Type}{\textbf{Type}}

\DeclareMathOperator{\refl}{\mathsf{refl}}
\DeclareMathOperator{\fib}{\mathsf{fib}}

\DeclareMathOperator{\fst}{\term{fst}}
\DeclareMathOperator{\snd}{\term{snd}}

\DeclareMathOperator{\ap}{\term{ap}}
\DeclareMathOperator{\at}{\term{\, at\,}}

\newcommand{\shape}{\mathbin{\textup{\textesh}}}

\newcommand{\sslash}{\mathbin{/\mkern-6mu/}} %

\DeclareMathOperator{\Vect}{\textbf{Vect}}

\newcommand{\lam}[1]{\lambda {#1}.\,}       %
\newcommand{\dprod}[1]{({#1}) \to }           %
\newcommand{\dsum}[1]{({#1}) \times }       %
\newcommand{\type}[1]{\mathsf{{#1}}}
\newcommand{\term}[1]{\mathsf{{#1}}}

\newcommand{\trunc}[1]{\left\lVert#1\right\rVert}       %
\newcommand{\BB}{\term{B}}
\newcommand{\pt}{\term{pt}}

\newcommand{\xto}[1]{\xrightarrow{#1}}

\newcommand{\pto}{\,\cdot\kern-.1em{\to}\,}
\newcommand{\pcirc}{\overset{.}{\circ}}

\makeatletter
\providecommand*{\xmapstofill@}{%
  \arrowfill@{\mapstochar\relbar}\relbar\rightarrow
}
\providecommand*{\xmapsto}[2][]{%
  \ext@arrow 0395\xmapstofill@{#1}{#2}%
}

\makeatletter
\def\slashedarrowfill@#1#2#3#4#5{%
  $\m@th\thickmuskip0mu\medmuskip\thickmuskip\thinmuskip\thickmuskip
   \relax#5#1\mkern-7mu%
   \cleaders\hbox{$#5\mkern-2mu#2\mkern-2mu$}\hfill
   \mathclap{#3}\mathclap{#2}%
   \cleaders\hbox{$#5\mkern-2mu#2\mkern-2mu$}\hfill
   \mkern-7mu#4$%
}
\def\rightslashedarrowfill@{%
  \slashedarrowfill@\relbar\relbar\mapstochar\rightarrow}
\newcommand\xslashedrightarrow[2][]{%
  \ext@arrow 0055{\rightslashedarrowfill@}{#1}{#2}}
\makeatother

\tikzset{ vert/.style={anchor=south, rotate=90, inner sep=.5mm} }

\newtheorem{thm}{Theorem}[section]

\theoremstyle{definition}
\newtheorem{defn}[thm]{Definition}

\newtheorem{axiom}{Axiom}

\newtheorem{rmk}[thm]{Remark}

\newtheorem*{acknowledgements}{Acknowledgements}

\newtheorem{lem}[thm]{Lemma}
\newtheorem*{notation}{Notation}

\newtheorem{cor}[thm]{Corollary}
\newtheorem{prop}[thm]{Proposition}



\begin{document}

\title{Good Fibrations through the Modal Prism}
\author{David Jaz Myers}
\date{\today}
\maketitle

\begin{abstract}
    Homotopy type theory is a formal language for doing abstract homotopy theory
    --- the study of identifications. But in unmodified homotopy type theory, there is no way to say that these identifications come from identifying the
    path-connected points of a space. In other words, we can do abstract
    homotopy theory, but not algebraic topology. Shulman's \emph{Real Cohesive
      HoTT} remedies this issue by introducing a system of modalities that
    relate the spatial structure of types to their homotopical structure. In
    this paper, we develop a theory of \emph{modal fibrations} for a general
    modality, and apply it in particular to the shape modality of real cohesion.
    We then give examples of modal fibrations in Real Cohesive HoTT, and develop
    the theory of covering spaces.
\end{abstract}

\tableofcontents

\section{Introduction}

While homotopy theory --- the study of identifications --- has been well developed in homotopy type theory, algebraic topology --- the study of the connectivity of space --- has been somewhat lacking. This is because Book HoTT (the homotopy type theory of the HoTT Book \cite{HoTTBook}) has no way of saying that a type is \emph{the homotopy type} of another type. While we can define both the homotopy circle $S^1$ as a higher inductive type and the topological circle
$$\Sb^{1} :\equiv \{(x,\, y) : \Rb^2 \mid x^2 + y^2 = 1\},$$
in Book HoTT alone we do not have the tools to say that $S^1$ is the homotopy type of $\Sb^1$.

In his Real Cohesive Homotopy Type Theory \cite{RealCohesion}, Shulman solves
this issue by adding a system of modalities which includes the \emph{shape}
modality $\shape$ that takes a type $X$ to its homotopy type $\shape
X$.\footnote{The symbol ``$\shape$'' is an \emph{esh}, the IPA symbol for the
 voiceless palato-alveolar fricative phoneme \emph{/sh/} that begins the word
 ``shape''. It is not an integral sign.} In Real Cohesive HoTT, every type has a
spatial structure and every map is continuous with respect to this spatial
structure. This spatial structure is distinct from the \emph{homotopical}
structure of identifications that every type has in homotopy type theory. But
these two structures are brought together by the $\shape$ modality, which allows
us to identify points by giving spatial paths between them. Formally, the $\shape$
modality is given by localizing at the type of Dedekind real numbers $\Rb$ --- in
other words, by identifying points which are connected by paths $\gamma
: \Rb \to X$.\footnote{In this
  paper, we reserve the term \emph{path} (in $X$) for function $\gamma : \Rb \to
  X$, while we use the term \emph{identification} for points of the type $x = y$
(for $x,\, y : X$). This conflicts with the terminology of the HoTT Book, in
which ``path'' is used for what we call identifications. But, in our setting, the shape modality $\shape$ takes a path $\gamma :
\Rb \to X$ and gives an identification $\gamma(0)^{\shape} = \gamma(1)^{\shape}$ in
the homotopy type $\shape X$. So, when one is working with homotopy types $\shape X$,
the difference between our terminology and the terminology of the HoTT Book is blurred.}

As with any modality, there is a \emph{modal unit} $(-)^{\shape} : X \to \shape X$,
a quotient map of sorts, which is the universal map from $X$ to a
\emph{discrete} type --- one with only homotopical and no spatial
structure.\footnote{That is, every path is constant in a discrete type, but
  there may still be non-trivial identifications between its points.} For any map $f : X \to Y$, we have a naturality square which induces a map from the fiber of $f$ over $y : Y$ to its \emph{homotopy fiber}, the fiber of $\shape f$:
\[
\begin{tikzcd}
\fib_f(y) \arrow[d] \arrow[r, dashed, "\delta"] & \fib_{\shape f}(y^{\shape}) \arrow[d] \\
X \arrow[d, "f"'] \arrow[r, "(-)^{\shape}"] & \shape X \arrow[d, "\shape f"]\\
Y \arrow[r, "(-)^{\shape}"'] & \shape Y
\end{tikzcd}
\]
The fibers of maps between discrete types are themselves discrete, so the map $\delta : \fib_f(y) \to \fib_{\shape f}(y^{\shape})$ factors uniquely through $(-)^{\shape} : \fib_f(y) \to \shape \fib_f(y)$ by the universal property of the unit. This gives us a useful diagram (Figure \ref{fig:modal.prism}) which I like to call \emph{the modal prism}.

\begin{figure}[h]
    \[\begin{tikzcd}
\fib_f(y) \arrow[rr, "\delta"] \arrow[dr, "(-)^{\shape}"']&  & \fib_{\shape f}(y^{\shape}) \\
& \shape \fib_f(y) \arrow[ur, "\gamma"'] &
\end{tikzcd}\]
\caption{The Modal Prism.}
\label{fig:modal.prism}
\end{figure}

Looking through the modal prism, we see a rainbow of different possibilities for a function $f : X \to Y$.
\begin{defn}
Let $f : X \to Y$ and consider the modal prism as in Figure 1. Then $f$ is
\begin{itemize}
    \item \emph{$\shape$-modal} if its fibers are discrete, that is, if $(-)^{\shape}$ is an equivalence for all $y : Y$,
    \item \emph{$\shape$-connected} if its fibers are homotopically contractible, that is, if $\shape \fib_{f}(y)$ is contractible for all $y : Y$,
    \item \emph{$\shape$-{\'e}tale} if its fibers are its homotopy fibers, that is, if $\delta$ is an equivalence for all $y : Y$.
    \item a \emph{$\shape$-equivalence} if its homotopy fibers are contractible, that is, if $\fib_{\shape f}(y^{\shape})$ is contractible for all $y : Y$,
    \item a \emph{$\shape$-fibration} if the homotopy type of its fibers are its homotopy fibers, that is, if $\gamma$ is an equivalence for all $y : Y$.
\end{itemize}
\end{defn}

For the shape modality, a map is modal when it has discrete fibers, and is a
modal equivalence, or (weak) homotopy equivalence, when it induces an
equivalence on homotopy types. It is modally connected when it has the stronger
property that its fibers are homotopically contractible; for comparison, consider the inclusion $x : \Rb \to \Rb^2$ of the $x$-axis, which is clearly a homotopy equivalence but is not $\shape$-connected since some of its fibers are empty. Finally, a $\shape$-{\'e}tale map is a weak relative of a covering map; it has a \emph{unique} lifting against any homotopy equivalence.

The notions of modal maps, connected maps, and modal equivalences appear in the
HoTT Book (\cite{HoTTBook}). For the $n$-truncation modality, these are
$n$-truncated and $n$-connected maps respectively, with modal equivalences not
given a specific name. The notion of modal {\'e}tale map is due to Wellen as a
``formally {\'e}tale map'' in \cite{WellenThesis}, building on work of Schreiber
in the setting of higher topos theory \cite{DiffCohomologyCohesiveTopos}. In the
case of $\shape$, it appears as a ``modal covering'' in \cite{WellenCovering}.

The notion of modality has also made its way into the $\infty$-categorical
literature through the work of Anel, Biederman, Finster, and Joyal (see
\cite{BlakersMasseyABFJ} and \cite{GoodwillieCalculusABFJ}). In these papers,
they define a modality as a \emph{stable orthogonal factorization system} (one
of the equivalent ways of defining a modality in HoTT), and translate a
homotopy type theoretic generalized Blakers-Massey Theorem into the language of
$\infty$-categories and apply it to the
Goodwillie calculus of functors. As Shulman has proven that every $\infty$-topos
models HoTT (\cite{ShulmanHoTTModels}), the results in
this paper concerning modal fibrations (in Section \ref{sec:modal.fibrations}) apply in any
$\infty$-topos as well.

The notion of modal fibration is, as far as I know, novel to this paper. It
gives a good notion of fibration in real cohesion which works not just for set
level spaces (e.g. manifolds) but also spaces with both topological and
homotopical content (e.g. orbifolds and Lie groupoids). A map is a
$\shape$-fibration when the homotopy type of its fibers are the fibers of its
action on homotopy types; this gives us the long fiber sequence on homotopy
groups we expect from a fibration in real cohesion. This definition closely resembles the classical notion of \emph{quasi-fibration} due to Dold and Thom \cite{DoldThomQuasifibration}, though it is much better behaved (see Remark \ref{rmk:quasifibration}).

In
Section \ref{sec:modality.refresher}, we will refresh ourselves on modalities and look through the modal
prism to see the different kinds of functions associated with a modality. Then,
in Section \ref{sec:modal.fibrations}
we will develop the basic theory of ${\modal}$-fibrations for an arbitrary modality
${\modal}$, and justify the name. In summary, the ${\modal}$-fibrations are closed under composition and pullback and
may be characterized in any one of the following ways.
\begin{thm}
  For a map $f : X \to Y$, the following are equivalent:
  \begin{enumerate}
  \item $f$ is a ${\modal}$-fibration.
  \item ${\modal}$ preserves all fibers of $f$.
  \item ${\modal}$ preserves all pullbacks along $f$.
  \item The ${\modal}$-connected/${\modal}$-modal and
    ${\modal}$-equivalence/${\modal}$-{\'e}tale factorizations of $f$ agree.
  \item The ${\modal}$-modal factor of $f$ is ${\modal}$-{\'e}tale.
  \item The ${\modal}$-equivalence factor of $f$ is ${\modal}$-connected.
  \item The $\modal$-naturality square of $f$ is $\modal$-cartesian.
  \item The connecting map $\textbf{tot}(\gamma)$ between the two factorizations
    of $f$ is a ${\modal}$-fibration.
    \item $f$ has ${\modal}$-locally constant ${\modal}$-fibers in the sense that
      ${\modal} \fib_f : Y \to \Type_{{\modal}}$ factors through ${\modal} Y$.
  \item (If $\modal$-units are surjective:) For every $x : X$, the induced map

    $\fib_{(-)^{{\modal}}}(x^{{\modal}}) \to \fib_{(-)^{{\modal}}}((f x)^{{\modal}})$ is ${\modal}$-connected.
  \end{enumerate}
\end{thm}
In particular, we will prove in Theorem
\ref{thm:fibration.iff.locallyconstant.modal.fibers} that a map $f : X \to Y$ is
an ${\modal}$-fibration if and only if the type family ${\modal} \fib_f : Y \to \Type$ factors
through the modal unit $(-)^{{\modal}} : Y \to {\modal} Y$. For the modality $\shape$,
this means that a map is a $\shape$-fibration if and only if the homotopy type of
its fiber over $y : Y$ is locally constant in $y$; that is, a map is a $\shape$-fibration if and only if its fibers form a \emph{local system} on its codomain.

We will also characterize the
$\trunc{-}_n$-fibrations as those maps which are surjective on $\pi_{n+1}$ in
Corollary \ref{cor:truncation.fibrations}.

In Section \ref{sec:cohesion.review}, we give a brief review of Shulman's
Real Cohesive HoTT. We then prove in Section
\ref{sec:discrete.classifying.types} that the classifying types of bundles of
discrete structures are themselves discrete (see Theorem
\ref{thm:locally.crisply.discrete.BAut.discrete} for the precise statement). As
a corollary, we find in Theorem \ref{thm:characterizing.shape.fibration} that
maps whose fibers have a merely constant homotopy type are $\shape$-fibrations. Morally,
this result says that if all the fibers of a map have the same homotopy type so
that one can comfortably write
$$F \to E \xto{p} B$$
with $F$ well defined up to homotopy, then $p$ is a $\shape$-fibration.

In the remaining sections, we will show how this theory can be applied to synthetic algebraic topology. Because the homotopy type of the fibers of a $\shape$-fibration are its homotopy fibers, whenever
$$F \to E \xto{p} B$$
is a fiber sequence with $p$ a $\shape$-fibration,
$\shape F \to \shape E \xto{\shape p} \shape B$
is also a fiber sequence. Using the fact that the fibers of the map $(\cos,\sin)
: \Rb \to \Sb^1$ are merely equivalent to $\Zb$, Theorem \ref{thm:characterizing.shape.fibration} implies that this map is a $\shape$-fibration, and that therefore,
$$\Zb \to \shape \Rb \to \shape \Sb^1$$
is a fiber sequence. Since $\shape \Rb \simeq \ast$ is contractible, this
calculates the loop space of the topological circle $\Sb^1$ without passing
through the higher inductive circle $S^1$. We consider this and other examples
of $\shape$-fibrations, including:
\begin{itemize}
\item The map $(\cos, \sin) : \Rb \to \Sb^1$ (in Section \ref{subsec:universal.cover.of.circle}).
  \item The homogeneous coordinates $\Sb^{n} \to \Rb P^n$, $\Sb^{2n+1} \to \Cb
    P^n$, and $\Sb^{4n + 3} \to \Hb P^n$, including as special cases the Hopf
    fibration $\Sb^3 \to \Cb P^1$ and the quaternionic Hopf fibration $\Sb^7 \to
    \Hb P^1$ (in Section \ref{subsec:hopf.fibrations}).
    \item The rotation map $\textbf{SO}(n + 1) \to \Sb^n$ (in Section \ref{subsec:rotation.of.spheres}).
    \item The homotopy quotient $\Rb \vee \Rb \to (\Rb \vee \Rb) \sslash C_2$,
      and many other homotopy quotients (in Section \ref{sec:wedge.fibration.example}).
\end{itemize}

After this, we prove some corollaries for the theory of higher groups in
Sections \ref{sec:higher.groups} and \ref{sec:connectedness}. We begin by
reviewing the definition of higher groups, and then show that the homotopy
quotient $X \to X \sslash G$ of a type by the action of a crisp higher group is
always a $\shape$-fibration. We then prove that $\shape$ preserves the
connectedness of crisp types, and conclude that the homotopy type of a higher
group is itself a higher group.

Finally, in Section \ref{sec:Covering.Theory}, we turn to the theory of covering
spaces. We define the notion of covering following Wellen \cite{WellenCovering}, and show that the type of coverings on
a type is equivalent to the type of actions of its fundamental groupoid on
discrete sets. We then show that every pointed type has a universal cover,
and prove that this universal cover has the expected universal property. We end
by showing that the universal cover of a higher group is a higher group.

\begin{notation}
  In this paper, we will use Agda-inspired notation for the dependent pair and dependent function types. For a type family $E : B \to \Type$, we write
  \begin{align*}
    \dsum{b : B} E(b) &\equiv \sum_{b : B} E(b),\mbox{ and} \\
    \dprod{b : B} E(b) &\equiv \prod_{b : B} E(b)
\end{align*}
for the dependent pair (or depedent sum) type and the dependent function (or product) type respectively. The elements of $\dsum{b : B} E(b)$ are pairs $(b, e)$ with $b : B$ and $e : E(b)$. The elements of $\dprod{b : B} E(b)$ are functions $b \mapsto f(b)$ with $f(b) : E(b)$ for $b : B$.
  \end{notation}

\begin{acknowledgements}
I would like to thank Felix Wellen for introducing me to the modal covering
story, and for many interesting conversations on the topic. I would also like to
thank Egbert Rijke for his work on modalities and for a fruitful discussion on
modal fibrations. I would also like to thank the reviewer whose helpful comments have improved numerous parts of the exposition and cleaned up the proofs of a few lemmas. And, crucially, I would like to thank Emily Riehl for her helpful comments and guidance during the drafting of this paper.
\end{acknowledgements}

\section{Modalities and the Modal Prism}\label{sec:modality.refresher}

A modality is a way of changing what it means for two elements of a type to be identified. To each type $X$, we associate a new type ${{\modal}}X$ and a function $(-)^{{\modal}} : X \to {{\modal}}X$. For two points $x,\, y : X$ to be identified by the modality then means that $x^{{\modal}} = y^{{\modal}}$ as elements of ${{\modal}}X$. Here are a few examples of modalities, with emphasis on those we will focus on in this paper.

   \begin{itemize}
   \item With the trivial modality ${{\modal}}X = \ast$, any two points are uniquely identified.
   \item With the $n$-truncation modality $\trunc{-}_n$, two points are identified by giving an $(n-1)$-truncated identification between them. The base case is $\trunc{X}_{-2} = \ast$, the trivial modality.
   \item With the \emph{shape} modality $\shape$, two points may be identified by
     giving a path between them (that is, a map from the real line $\Rb$ which
     sends $0$ to one point and $1$ to the other). We call $\shape X$ the
     \emph{homotopy type} of a type $X$.\footnote{The modality $\shape$ appears
       as Definition 9.6 of \cite{RealCohesion}, and we review it in Section \ref{sec:cohesion.review}.}
   \item With the \emph{crystalline} modality $\mathfrak{I}$, two points may be
     identified by giving an \emph{infinitesimal} path between them. We call
     $\mathfrak{I} X$ the \emph{de Rham stack} of a type $X$.\footnote{The
       crystaline modality appears formally as Axiom 3.4.1 in
       \cite{WellenThesis}, and in the higher categorical setting in Definition
       4.2.1 of \cite{DiffCohomologyCohesiveTopos}, where it is called the
       \emph{infinitesimal shape} modality}
   \end{itemize}

   While the elementary theory of modalities appeared in the HoTT Book \cite{HoTTBook}, the notion was developed more fully by Rijke, Shulman, and Spitters in \cite{RSS}. In that paper, they give equivalences between four different notions of modality and prove a number of useful lemmas along the way. We will take our modalities to be ``higher modalities'', one of the many equivalent notions of modality.
 \begin{defn}
   A \emph{higher modality} consists of a \emph{modal operator} ${{\modal}} : \Type
   \to \Type$ together with:
   \begin{itemize}
   \item For each type $X$, a \emph{modal unit} $$(-)^{{\modal}}
   : X \to {{\modal}}X$$

   \item For every $A : \Type$ and $P : {{\modal}}A \to \Type$, an \emph{induction principle}
     $$\textbf{ind}^{{\modal}}_A : \big( \dprod{a : A} {{\modal}}P(a^{{\modal}}) \big) \to \big(\dprod{u : {{\modal}}A} {{\modal}}P(u) \big),$$
   \item For every $A : \Type$, $P : {\modal} A \to \Type$, $f : \dprod{a : A} \modal P(a^{\modal})$ and $x : A$, a \emph{computation rule}
     $$\textbf{comp}^{{\modal}}_A : \textbf{ind}^{{\modal}}_A(f)(x^{{\modal}}) = f(x),$$
   \item For any $u,\, v : {{\modal}}A$, a witness that the modal unit $(-)^{{\modal}} : u = v \to {{\modal}}(u = v)$ is an equivalence.
   \end{itemize}

   We say a type $X$ is \emph{${\modal}$-modal} if $(-)^{{\modal}} : X \to {\modal} X$ is an equivalence, and we define $$\Type_{{\modal}} :\equiv \dsum{X : \Type}{\term{is{\modal}Modal}(X)}$$ to be the universe of ${\modal}$-modal types. A type $X$ is \emph{${\modal}$-separated} if for all $x,\, y : X$, the type of identifications $x = y$ is ${\modal}$-modal.
\end{defn}

A modality is in particular a \emph{reflective subuniverse}: pre-composition by $(-)^{{\modal}}$ gives an equivalence
$$({\modal} X \to Z) \xto{\sim} (X \to Z)$$
whenever $Z$ is ${\modal}$-modal (see Theorem 1.13 of \cite{RSS}). Any map $\eta : X \to K$ from $X$ to a modal type $K$ which satisfies the same property is called a \emph{${\modal}$-unit}, since from this property it can be show that $K \simeq {\modal} X$ and $\eta = (-)^{{\modal}}$ under this equivalence.

Modal types are closed under the basic operations of dependent type theory in the following way.
\begin{lem}\label{lem:modality.facts}
Let $X$ be a type and $P : X \to \Type$ a family of types.
\begin{itemize}
    \item If $X$ is modal and for all $x : X$, $Px$ is modal, then $\dsum{x : X} Px$ is modal.
    \item If for all $x : X$, $Px$ is modal, then
    $\dprod{x : X} Px$ is modal.
\end{itemize}
\end{lem}
\begin{proof}
See Theorem 1.32 and Lemma 1.26 of \cite{RSS}.
\end{proof}

As a corollary, a number of useful properties of modal types are also modal.
\begin{cor}
Let $A$ be a modal type. Then
$$\type{isContractible}(A) :\equiv \dsum{a : A}\big(\dprod{a' : A} (a = a')\big)$$
is modal. If $B$ is also a modal type and $f : A \to B$, then
$$\type{isEquiv}(f) :\equiv \dprod{b : B}\type{isContractible}(\fib_f(b))$$
is modal.
\end{cor}

When we use the induction principle of a modality, it often makes sense to think of it ``backwards''. That is, we think of the induction principle as saying that in order to map out of ${\modal} A$ into a modal type, it suffices to map out of $A$. Or, with variables, in order to define $T(u) : {\modal} P(u)$ for $u : {\modal} A$, it suffices to assume that $u \equiv a^{{\modal}}$ for $a : A$. In prose, we will just say that ${ \modal }$-induction lets us assume $u$ is of the form $a^{{\modal}}$.

 We can extend the operation of ${\modal}$ to a functor using the induction
 principle. If $f : X \to Y$, then define ${\modal} f : {\modal} X \to {\modal} Y$ by
 ${\modal} f (x^{{\modal}}) :\equiv f(x)^{{\modal}}$, or explicitly by
$${\modal} f :\equiv\, \textbf{ind}^{{\modal}}_X((-)^{{\modal}} \circ f).$$

Using the computation rule, we get a \emph{naturality square}
\[
\begin{tikzcd}
X \arrow[d, "f"'] \arrow[r, "(-)^{{\modal}}"] & {\modal} X \arrow[d, "{\modal} f"] \\
Y \arrow[r, "(-)^{{\modal}}"'] & {\modal} Y
\end{tikzcd}
\]

Any commuting square induces a map from the fiber of the left map to the fiber
of the right. Therefore, we get the map $\delta : \fib_f(y) \to \fib_{{\modal}
  f}(y^{{\modal}})$ for any $y : Y$ given by
$$\delta((x : X), (p : fx = y)) :\equiv (x^{{\modal}}, \textbf{comp}^{{\modal}} \cdot (\ap\,
(-)^{{\modal}}\, p)).$$
As the sum of modal types is modal, $\fib_{{\modal} f}(y^{{\modal}}) \equiv \dsum{u : {\modal} X} ({\modal} f (u) = y^{{\modal}})$ is modal. Therefore, this map factors through ${\modal} \fib_f(y)$ uniquely, giving us the \emph{modal prism}.

    \[\begin{tikzcd}
\fib_f(y) \arrow[rr, "\delta"] \arrow[dr, "(-)^{{\modal}}"']&  & \fib_{{\modal} f}(y^{{\modal}}) \\
& {\modal} \fib_f(y) \arrow[ur, "\gamma"'] &
\end{tikzcd}\]

 The modal prism divides functions in 5 possible kinds. Four of these possibilities arrange themselves into orthogonal factorization systems; the other gives a mediating notion which is the focus of this paper.
\begin{defn}
Let $f : X \to Y$ and consider the modal prism as in Figure 1. Then $f$ is
\begin{itemize}
    \item \emph{${\modal}$-modal} if $(-)^{{\modal}}$ is an equivalence for all $y : Y$,
    \item \emph{${\modal}$-connected} if ${\modal} \fib_{f}(y)$ is contractible for all $y : Y$,
    \item \emph{${\modal}$-{\'e}tale} if $\delta$ is an equivalence for all $y : Y$.
    \item a \emph{${\modal}$-equivalence} if $\fib_{{\modal} f}(y^{{\modal}})$ is contractible for all $y : Y$,
    \item a \emph{${\modal}$-fibration} if $\gamma$ is an equivalence for all $y : Y$.
\end{itemize}
\end{defn}
\begin{rmk}
By a quick application of ${\modal}$-induction, we see that $f$ is a ${\modal}$-equivalence if and only if ${\modal} f$ is an equivalence. And, by the lemma that a square is a pullback if and only if the induced map on fibers is an equivalence, $f$ is ${\modal}$-{\'e}tale if and only if its naturality square is a pullback.
\end{rmk}

\noindent We can see relations between these definitions right off the bat.
\begin{lem}
Let $f : X \to Y$. Then:
\begin{itemize}
    \item $f$ is ${\modal}$-{\'e}tale if and only if it is ${\modal}$-modal and a ${\modal}$-fibration.
    \item $f$ is ${\modal}$-connected if and only if it is a ${\modal}$-equivalence and a ${\modal}$-fibration.
\end{itemize}
\end{lem}
\begin{proof}
Since the modal prism commutes, if $f$ is ${\modal}$-modal and a ${\modal}$-fibration, then it is ${\modal}$-{\'e}tale. On the other hand, since $\fib_{{\modal} f}(y^{{\modal}})$ is modal, if $f$ is ${\modal}$-{\'e}tale then $\fib_f(y)$ is ${\modal}$-modal and so $(-)^{{\modal}}$ is an equivalence and hence so is $\gamma$.

If $f$ is a ${\modal}$-equivalence and a ${\modal}$-fibration, then ${\modal} \fib_f(y)$
is contractible as it is equivalent to the contractible $\fib_{{\modal}
  f}(y^{{\modal}})$. On the other hand, if $f$ is ${\modal}$-connected, then it is a
${\modal}$-equivalence by Lemma 1.35 of \cite{RSS}, and so $\gamma$ is a map between contractible types and is therefore an equivalence.
\end{proof}

Recall that any function $f : X \to Y$ gives an equivalence $X \simeq \dsum{y : Y}{\fib_f(y)}$ over $Y$. Therefore, by totalizing the modal prism, we can find two factorizations of any map $f$, connected in the middle by $\term{tot}(\gamma)$:
\[
\begin{tikzcd}
                                                                               & X \arrow[ld, "\term{tot}((-)^{{\modal}})"'] \arrow[rd, "\term{tot}(\delta)"] \arrow[dd, crossing over, "f" description] &                                                         \\
                                                                               \dsum{y : Y} {\modal} \fib_f(y) \arrow[rd, "\fst"'] \arrow[rr, dashed, bend left=10, "\term{tot}(\gamma)"' pos=1.0] &                                                                                                       & \dsum{y: Y}\fib_{{\modal} f}(y^{{\modal}}) \arrow[ld, "\fst"] \\
                                                                               & Y                                                                                                     &
\end{tikzcd}
\]

In \cite{RSS}, Rijke, Shulman, and Spitters prove that the left factorization is a \emph{stable orthogonal factorization system}. In particular, $\term{tot}((-)^{{\modal}})$ is ${\modal}$-connected, and $\fst : \dsum{y : Y}{\modal} \fib_f(y) \to Y$ is ${\modal}$-modal, and these give the unique ${\modal}$-connected/${\modal}$-modal factorization of $f$. The connected/modal factorization of a map $f$ is also preserved under pullback; if $y : A \to Y$ is any map, then the factorization of the pullback $y^{\ast}f$ is the pullback of the factorization of $f$ along $y$.

This can be seen most clearly by viewing the factorization system from the point of view of type families. A map $f : X \to Y$ corresponds to the type family $\fib_f : Y \to \Type$, and its modal factor corresponds to the type family ${\modal} \fib_f : Y \to \Type$. On type families, pullback along $y : A \to Y$ corresponds to composition, so $y^{\ast}f$ corresponds to $\lam{a : A}\fib_f(y a) : A \to \Type$. The modal factorization of the pullback $y^{\ast}$ is then $\lam{a : A} {\modal} \fib_f(y a)$, which is precisely the pullback of the modal factorization of $f$.

In his thesis \cite{RijkeThesis}, Rijke proves that the right factorization is
an \emph{orthogonal factorization system}. In particular, $\term{tot}(\delta)$
is a ${\modal}$-equivalence and $\fst : \dsum{y : Y}{\fib_{{\modal} f}(y^{{\modal}})} \to
Y$ is ${\modal}$-{\'e}tale, and this is the unique
${\modal}$-equivalence/${\modal}$-{\'e}tale factorization of $f$. This is, however,
not a stable factorization system because the ${\modal}$-equivalences are not in
general preserved under pullback (see Remark
\ref{rmk:modal.equiv.not.stable.under.pullback} for an example).

Another important concept in the theory of modalities is that of a \emph{$\modal$-cartesian square} (see, for example, Definition 3.7.1 of \cite{BlakersMasseyABFJ}). We will make use of $\modal$-cartesian squares in developing the theory of modal fibrations, so we will establish a few lemmas here.

\begin{defn}
  A commuting square
  \[
    \begin{tikzcd}
      A \ar[r, "g"] \ar[d, "f"'] & B \ar[d, "h"] \\
      C \ar[r, "k"'] & D
      \end{tikzcd}
  \]
  is \emph{$\modal$-cartesian} if the cartesian gap map $A \to B \times_{D} C$ is $\modal$-connected.
  \end{defn}
  Note that a $\id$-cartesian square for the identity modality $\id$ is simply a pullback. Before proving our lemmas concerning $\modal$-cartesian squares,
    \begin{lem}\label{lem:fibers.of.gap.map.are.fibers.of.fibers}
  Consider a square
  \[
\begin{tikzcd}
	A & B \\
	C & D
	\arrow["f"', from=1-1, to=2-1]
	\arrow["g", from=1-1, to=1-2]
	\arrow["h", from=1-2, to=2-2]
	\arrow["k"', from=2-1, to=2-2]
\end{tikzcd}
  \]
  commuting via $S : \dprod{x : A} (k(f(x)) = h(g(x)))$. Let $c : C$, and define the map $G : \fib_{f}(c) \to \fib_{h}(k c)$ by
  $$G(x : A,\, w : f x = c) :\equiv (g x,\, S(x)\inv \cdot k_{\ast}w).$$
  Then for any $(b, p) : \fib_{h}(k c)$, we have an equivalence $\fib_{G}((b, p)) = \fib_{\text{gap}}((c, b p))$ with the fiber of the gap map $A \to B \times_{D} C$.
    \end{lem}
    \begin{proof}
    We find the equivalence as the following composite:
    \begin{align*}
      \fib_{G}((b, p)) :&\equiv \dsum{(x, w) : \fib_{g}(c)} (G(x, w) = (b, p)) \\
                        &= \dsum{x : A} \dsum{w : f x = c} ((g x, S(x)\inv \cdot k_{\ast} w) = (b, p)) \\
      &= \dsum{x : A} ((gx, fx, S(x)\inv) = (b , c, p ))\\
      &= \fib_{\text{gap}}((b, c , p)).
    \end{align*}
    \end{proof}

    Using this, we can give a characterization of $\modal$-cartesian maps which resembles the usual characterization of pullbacks as fiberwise equivalences.
  \begin{lem}\label{lem:cartesian.if.connected.on.fibers}
A commuting square
  \[
    \begin{tikzcd}
      A \ar[r, "g"] \ar[d, "f"'] & B \ar[d, "h"] \\
      C \ar[r, "k"'] & D
      \end{tikzcd}
  \]
  is $\modal$-cartesian if and only if for every $c : C$, the induced map
  $$G : \fib_{f}(c) \to \fib_{h}(kc)$$
  induced on fibers is $\modal$-connected.
    \end{lem}
    \begin{proof}
By Lemma \ref{lem:fibers.of.gap.map.are.fibers.of.fibers}, the fibers of the gap map are the fibers of $G$; so, the fibers of the gap map are $\modal$-connected if and only if the fibers of $G$ are.
    \end{proof}

    The following lemmas may be found in \cite{BlakersMasseyABFJ} as Lemmas 3.7.4 and 3.7.3 respectively. We will prove them in HoTT.
    \begin{lem}\label{lem:composing.cartesian.squares}
      Consider a pair of commuting squares:
      \[
\begin{tikzcd}
	A & B & E \\
	C & D & F
	\arrow[from=1-1, to=2-1]
	\arrow[from=1-1, to=1-2]
	\arrow[from=1-2, to=2-2]
	\arrow["k"', from=2-1, to=2-2]
	\arrow[from=1-2, to=1-3]
	\arrow[from=1-3, to=2-3]
	\arrow[from=2-2, to=2-3]
\end{tikzcd}
      \]
      Then
      \begin{enumerate}
\item If the left square and the right square are $\modal$-cartesian, then so is the composite square.
              \item If the left square and the composite square are $\modal$-cartesian, and $k$ is surjective, then the right square is $\modal$-cartesian.
              \item If the right square is a pullback and the composite square is $\modal$-cartesian, then the left square is $\modal$-cartesian.
      \end{enumerate}
    \end{lem}
    \begin{proof}
We will appeal to Lemma \ref{lem:cartesian.if.connected.on.fibers} a number of times. To prove the first fact, let $c : C$ and consider the following diagram:
\[
\begin{tikzcd}
	{\fib_f(c)} & {\fib_h(kc)} & {\fib_{\ell}(jkc)} \\
	A & B & E \\
	C & D & F
	\arrow["f"', from=2-1, to=3-1]
	\arrow[from=2-1, to=2-2]
	\arrow["h", from=2-2, to=3-2]
	\arrow["k"', from=3-1, to=3-2]
	\arrow[from=2-2, to=2-3]
	\arrow["\ell", from=2-3, to=3-3]
	\arrow["j"', from=3-2, to=3-3]
	\arrow[from=1-1, to=1-2]
	\arrow[from=1-2, to=1-3]
	\arrow[from=1-1, to=2-1]
	\arrow[from=1-2, to=2-2]
	\arrow[from=1-3, to=2-3]
\end{tikzcd}
\]
The squares are $\modal$-cartesian when the maps on fibers are $\modal$-connected, and $\modal$-connected maps are closed under composition, so the outer square is also $\modal$-cartesian.

With a modification of the above argument, we can prove the third fact. Suppose instead that the right square is a pullback, so that $\fib_{h}(kc) \to fib_{\ell}(jkc)$ is an equivalence. Then since the composite map $\fib_{f}(c) \to \fib_{\ell}(jkc)$ is $\modal$-connected, so is $\fib_{f}(c) \to \fib_{h}(kc)$.

To prove the second fact, suppose that $d : D$; then, since $k$ is assumed to be surjective and we are trying to prove a proposition, we may suppose we have a $c : C$ with $kc = d$. Then we can consider the above diagram again with $\fib_{f}(c) \to \fib_{h}(d)$ and $\fib_{f}(c) \to \fib_{\ell}(jd)$ modally connected. By right cancellability of modally connected maps (Lemma 1.33 of \cite{RSS}), we see that therefore $\fib_{h}(d) \to \fib_{\ell}(jd)$ is $\modal$-connected.

    \end{proof}

    \begin{lem}\label{lem:modal.square.modal.factorization}
      Suppose that
  \[
    \begin{tikzcd}
      A \ar[r, "g"] \ar[d, "f"'] & B \ar[d, "h"] \\
      C \ar[r, "k"'] & D
      \end{tikzcd}
  \]
  is a $\modal$-cartesian square. In its modal factorization
  \begin{equation}\label{diag:factored.square.lemma}
\begin{tikzcd}
	A & {\dsum{b : B} \modal\fib_{g}(b)} & B \\
	C & {\dsum{d : D} \modal\fib_{k}(d)} & D
	\arrow[from=1-1, to=2-1]
	\arrow[from=1-1, to=1-2]
	\arrow[from=1-2, to=1-3]
	\arrow[from=2-1, to=2-2]
	\arrow[from=2-2, to=2-3]
	\arrow[from=1-3, to=2-3]
	\arrow[from=1-2, to=2-2]
\end{tikzcd}
\end{equation}
  the right square is a pullback.
    \end{lem}
    \begin{proof}
      Here we will use the proof of this fact from Lemma 3.7.3 of \cite{BlakersMasseyABFJ}. Consider the following diagram:
      \[\begin{tikzcd}
	A & {B \times_DC} & {B \times_D(\dsum{d: D}\modal\fib_k(d))} & B \\
	& C & {\dsum{d : D} \modal \fib_k(d)} & D
	\arrow["f"', from=1-1, to=2-2]
	\arrow["y", from=1-3, to=1-4]
	\arrow["\ell"', from=2-2, to=2-3]
	\arrow["r"', from=2-3, to=2-4]
	\arrow["h", from=1-4, to=2-4]
	\arrow[from=1-3, to=2-3]
	\arrow["x", from=1-2, to=1-3]
	\arrow[from=1-1, to=1-2]
	\arrow[from=1-2, to=2-2]
\end{tikzcd}
      \]
      where we have taken two pullbacks. By construction, $\ell$ is $\modal$-connected and $r$ is $\modal$-modal. By stability of the $\modal$-connected / $\modal$-modal factorization system, $x$ is also $\modal$-connected and $y$ is $\modal$-modal. Since by hypothesis the gap map $A \to B \times_{D} C$ is $\modal$-connected, the composite $A \to B \times_{D}(\dsum{d : D} \modal \fib_{k}(d))$ is $\modal$-connected, so by the uniqueness of $\modal$-connected / $\modal$-modal factorizations, we see that $B \times_{D} (\dsum{d :D}\modal\fib_{k}(d))$ must be equivalent to $\modal$-factorization $\dsum{ b : B } \modal \fib_{g} (b)$. Therefore, the right hand pullback square in the above diagram is equivalent to the right hand square in Diagram \ref{diag:factored.square.lemma}, showing that it is a pullback.
    \end{proof}

Using these lemmas, we can prove a slight improvement of the Proposition 5.1 of \cite{Rijke-Cherubini:Modal.Descent}, using essentially the same proof.
\begin{thm}\label{thm:modal.descent}
Suppose that
  \[
    \begin{tikzcd}
      A \ar[r, "g"] \ar[d, "f"'] & B \ar[d, "h"] \\
      C \ar[r, "k"'] & D
      \end{tikzcd}
  \]
  is a $\modal$-cartesian square, and that $B$ and $D$ are $\modal$-modal. Then the square
  \[
    \begin{tikzcd}
      \modal A \ar[r, "\tilde{g}"] \ar[d, "\modal f"'] & B \ar[d, "h"] \\
      \modal C \ar[r, "\tilde{k}"'] & D
      \end{tikzcd}
  \]
  is a pullback, where the maps $\tilde{g} : \modal A \to B$ and $\tilde{k} : \modal C \to D$ are the unique factorizations of $g$ and $k$ respectively.
\end{thm}
\begin{proof}
  Consider the following diagram:
  \begin{equation}\label{diag:modal.descent}
\begin{tikzcd}
	A & {B \times_D \modal C} & B \\
	C & {\modal C} & D
	\arrow["f"', from=1-1, to=2-1]
	\arrow["{(-)^{\modal}}"', from=2-1, to=2-2]
	\arrow["{\tilde{k}}"', from=2-2, to=2-3]
	\arrow["h", from=1-3, to=2-3]
	\arrow[from=1-2, to=2-2]
	\arrow["x", from=1-2, to=1-3]
	\arrow["{r}", from=1-1, to=1-2]
\end{tikzcd}
\end{equation}
  We will start by showing that the map $r : A \to B \times_{D} \modal C$ is $\modal$-connected. Let $c : C$, and extend the diagram as follows:
  \[
\begin{tikzcd}
	{\fib_f(c)} & \fib_{\snd}(c^{\modal}) & {\fib_h(kc)} \\
	A & {B \times_D \modal C} & B \\
	C & {\modal C} & D
	\arrow["f"', from=2-1, to=3-1]
	\arrow["{(-)^{\modal}}"', from=3-1, to=3-2]
	\arrow["{\tilde{k}}"', from=3-2, to=3-3]
	\arrow["h", from=2-3, to=3-3]
	\arrow[from=2-2, to=3-2]
	\arrow["x", from=2-2, to=2-3]
	\arrow["{r}", from=2-1, to=2-2]
	\arrow["z", from=1-1, to=1-2]
	\arrow["\sim", from=1-2, to=1-3]
	\arrow[from=1-2, to=2-2]
	\arrow[from=1-3, to=2-3]
	\arrow[from=1-1, to=2-1]
\end{tikzcd}
  \]
  Since the square on the bottom right is a pullback, we get and equivalence between the map $z : \fib_{f}(c)\to \fib_{\snd}(c^{\modal})$ and the composite $G : \fib_{f}(c) \to \fib_{h}(kc)$. Since, by Lemma \ref{lem:cartesian.if.connected.on.fibers}, $G$ is $\modal$-connected, we see for all $c : C$ the map $z : \fib_{f}(c) \to \fib_{\snd}(c^{\modal})$ is $\modal$-connected. Since $(-)^{\modal}$ is always $\modal$-connected, we may conclude by Lemma 1.39 of \cite{RSS} that the map $r : A \to B \times_{D} \modal C$ is $\modal$-connected.

Now, as the pullback of maps between modal types, $B \times_{D} \modal C$ is modal. Therefore, $r$ is a $\modal$-connected map into a $\modal$-modal type, which makes it a $\modal$-unit. Therefore, the square on the right in Diagram \ref{diag:modal.descent} is the square we are trying to show is a pullback.
  \end{proof}
  \begin{rmk}
We can also see Theorem \ref{thm:modal.descent} as a corollary of Lemma \ref{lem:modal.square.modal.factorization} by noting that the right square in that lemma will be the square in the conclusion of Theorem \ref{thm:modal.descent} when $B$ and $D$ are modal.
  \end{rmk}

\section{Modal Fibrations}\label{sec:modal.fibrations}

Recall that a map $f : X \to Y$ is a \emph{${\modal}$-fibration} if and only if the induced map $\gamma : {{\modal}}\fib_f(y) \to \fib_{{{\modal}}f}(y^{{\modal}})$ is an equivalence for all $y : Y$.
In other words, $f : X \to Y$ is a ${{\modal}}$-fibration if ${\modal}$ preserves its fibers in the sense that whenever
$$F \to X \xto{f} Y$$
is a fiber sequence (for any pointing of $Y$), so is
$${{\modal}}F \to {{\modal}}X \xto{{{\modal}}f} {{\modal}}Y.$$
In other words, a ${{\modal}}$-fibration is a map $f$ whose fibers ``correctly represent'' the fibers of ${{\modal}}f$.

For example, consider the shape modality $\shape$. A $\shape$-fibration is a map $f
: X \to Y$ whose fibers have the same homotopy type as its homotopy fibers, the
fibers of its induced map $\shape f : \shape X \to \shape Y$ on homotopy types. An simple example of a $\shape$-fibrations is the projection $\pi_{1} : \Rb^{3} \to \Rb^{2}$; all the fibers of this map are identifiable with $\Rb$ whose shape is contractible, and the fibers of its induced map on homotopy types are contractible. An
example of a map which isn't a fibration is the inclusion $i : \ast \to \Rb^2$
of the origin into the real plane. Over the point $(1, 1) : \Rb^2$, the fiber of
$i$ is empty, and so its homotopy type is empty. But the induced map $\shape i :
\shape \ast \to \shape \Rb^2$ is an equivalence since $\shape \Rb^2$ is contractible, and so all the fibers of $\shape i$ are equivalent to $\ast$ which is not empty.

\begin{rmk}\label{rmk:quasifibration}
This is the sense in which a ${\modal}$-fibration is a ``fibration''. It most closely resembles the notion of \emph{quasi-fibration} of topological spaces introduced by Dold and Thom in \cite{DoldThomQuasifibration}, which is a continuous map $f : X \to Y$ such that for all $y \in Y$, the canonical map from the inverse image $f\inv(y)$ to the homotopy fiber $\fib_f(y)$ is a weak equivalence. If, seeking analogy, we take ``weak equivalence'' to be ${\modal}$-equivalence (which, for $\shape$, means that a map is a weak equivalence if it induces an equivalence on homotopy types), then a ${\modal}$-fibration is map $f$ whose fibers are weakly equivalent to its ``modal fibers'', the fibers of ${\modal} f$.

However, the notion of ${\modal}$-fibration is somewhat more robust than the notion of quasi-fibration, even in the case of $\shape$. As we will see, ${\modal}$-fibrations are closed under pullback, while quasi-fibrations are not. In this sense, $\modal$-fibrations more closely resemble the \emph{universal quasi-fibrations} introduced by Goodwillie in an email to the ALGTOP mailing list \cite{GoodwillieUniversalQuasifibration}. Intuitively, this is because universal quantification in type theory says more than it does in set theory --- it implies a sort of continuity. We will come back to this subtle point in the next section when we introduce the notion of a crisp variable from Shulman's real hohesion \cite{RealCohesion} in order to give a trick for showing a map is a $\shape$-fibration.
\end{rmk}

Before we get there, let's develop the basic theory of ${\modal}$-fibrations for a general modality. First, we will characterize ${\modal}$-fibrations as those maps on which the two factorization systems of ${\modal}$ agree.
\begin{lem}\label{lem:fibration.iff.factor.etale}
For $f : X \to Y$, the following are equivalent:
\begin{enumerate}
    \item $f$ is a ${\modal}$-fibration.
    \item The ${\modal}$-modal factor of $f$ is ${\modal}$-{\'e}tale.
    \item The ${\modal}$-equivalence factor of $f$ is ${\modal}$-connected.
    \item The ${\modal}$-connected/${\modal}$-modal and ${\modal}$-equivalence/${\modal}$-{\'e}tale factorizations of $f$ are equal as factorizations of $f$.
        \item The $\modal$-naturality square for $f$ is $\modal$-cartesian.
\end{enumerate}
\end{lem}
\begin{proof}
We will first show that the first two conditions are equivalent; then we will argue that the next three are all equivalent by the uniqueness of each factorization. Finally, we note that the last condition is immediately equivalent to the third, since the $\modal$-equivalence factor of $f$ is the gap map of the $\modal$-naturality square.

By Lemma 1.24 of \cite{RSS}, the unique factorization of the map $$\lam{(y,x)} (y, x^{{\modal}})^{{\modal}} : \dsum{y : Y} \fib_f(y) \to {\modal}(\dsum{y : Y}{{\modal} \fib_f(y)})$$ through ${\modal}(\dsum{y : Y} \fib_f(y))$ is an equivalence. Therefore, the composite
$$\dsum{y : Y} {\modal} \fib_f(y) \xto{(-)^{{\modal}}} {\modal}(\dsum{y : Y}{{\modal} \fib_f(y)}) \xto{\sim} {\modal}(\dsum{y : Y} \fib_f(y))$$
is a ${\modal}$-unit. So, for any $y : Y$, we get a diagram

    \[
    \begin{tikzcd}
    \fib_f(y) \arrow[d] \arrow[r] & {{\modal}}\fib_f(y) \arrow[d] \arrow[r, "\gamma"] & \fib_{{{\modal}}f}(y^{{\modal}}) \arrow[d] \\
    X \arrow[r] \arrow[d, "f"']         & \dsum{y : Y} {\modal} \fib_f(y) \arrow[r] \arrow[d]      & {{\modal}}X \arrow[d]             \\
    Y \arrow[r, "\id"']                   & Y \arrow[r]                    & {{\modal}}Y
    \end{tikzcd}\]
    in which the bottom right square is a ${\modal}$-naturality square. The map $f$ is a ${{\modal}}$-fibration if and only if the connecting map $\gamma$ is an equivalence for all $y : Y$, and this happens if and only if the bottom right square is a pullback. But the bottom right square is a pullback precisely when $\fst : \dsum{y : Y} {\modal} \fib_f(y)\to Y$ is ${{\modal}}$-{\'e}tale.

On the other hand, the fourth condition implies the second and third by simply transporting the properties. Each of the second and third also imply the fourth by the uniqueness of each factorization. Without loss of generality, consider the second condition. The ${\modal}$-connected factor of $f$ is always a ${\modal}$-equivalence, so if the modal factor of $f$ is ${\modal}$-{\'e}tale then the ${\modal}$-connected/${\modal}$-modal factorization is a ${\modal}$-equivalence/${\modal}$-{\'e}tale factorization and so is equal to the canonical one by the uniqueness of such factorizations.
\end{proof}

As a corollary, we can prove that ${\modal}$-fibrations are closed under pullback, and give a descent theorem for $\modal$-fibrations.
\begin{cor} \label{cor:fibrations.closed.under.pullback}
Let
\[
\begin{tikzcd}
A \arrow[d, "g"'] \arrow[r, "x"] & X \arrow[d, "f"] \\
B \arrow[r, "y"'] & Y
\end{tikzcd}
\]
be a $\modal$-cartesian square. If $f$ is a fibration, then so is $g$. In particular, $\modal$-fibrations are closed under pullback.
\end{cor}
\begin{proof}
Consider the following cube:
\begin{equation} \label{diag:fibration.cube}
\begin{tikzcd}
	& {\modal A} && {\modal X} \\
	A && X \\
	& {\modal B} && {\modal Y} \\
	B && Y
	\arrow[from=2-1, to=1-2]
	\arrow[from=4-1, to=3-2]
	\arrow["g"', from=2-1, to=4-1]
	\arrow[from=1-2, to=3-2]
	\arrow[from=1-2, to=1-4]
	\arrow["y"', from=4-1, to=4-3]
	\arrow[from=3-2, to=3-4]
	\arrow["{\modal f}", from=1-4, to=3-4]
	\arrow[from=4-3, to=3-4]
	\arrow[from=2-3, to=1-4]
	\arrow[from=4-1, to=3-4]
	\arrow[from=2-1, to=2-3, crossing over]
	\arrow[from=2-3, to=4-3, crossing over]
	\arrow[from=2-1, to=1-4, crossing over]
\end{tikzcd}
\end{equation}
By hypothesis, the front face is $\modal$-cartesian and, since $f$ is a $\modal$-fibration, so is the rightmost face. Therefore, by Lemma \ref{lem:composing.cartesian.squares}, the diagonal square is $\modal$-cartesian. Then, by Theorem \ref{thm:modal.descent}, the back face is a pullback. Then, by Lemma \ref{lem:composing.cartesian.squares} again, the leftmost face is $\modal$-cartesian, which shows that $g$ is a $\modal$-fibration.
  \end{proof}

\begin{rmk}
It is at this point that we require a full modality, rather than just a
\emph{reflective subuniverse}. The proof of Theorem \ref{thm:modal.descent} uses the fact that $\modal$-units are $\modal$-connected, a fact which characterizes modalities amongst localizations (also known as reflective subuniverses).
However, if
one could prove Theorem \ref{thm:modal.descent} without using this fact, or prove that the pullback of a ${\modal}$-{{\'e}tale} map is
${\modal}$-{{\'e}tale} for ${\modal}$ a reflective subuniverse, then we could prove the pullback stability of $\modal$-fibrations and so the rest of the
theory of ${\modal}$-fibrations would go through as well.
\end{rmk}

Using Lemma \ref{lem:composing.cartesian.squares} and the characterization of $\modal$-fibrations as those maps whose naturality squares are $\modal$-cartesian, we can show that $\modal$-fibrations have the same closure properties as $\modal$-cartesian squares.

\begin{thm}
  Let $f : X \to Y$ and $g : Y \to Z$ be maps.
  \begin{enumerate}
    \item If $f$ and $g$ are $\modal$-fibrations, then $g \circ f$ is a $\modal$-fibration.
    \item If $f$ and $g \circ f$ are $\modal$-fibrations, and $\modal f$ is surjective, then $g$ is a $\modal$-fibration.
    \item If $g$ is $\modal$-{\'e}tale and $g \circ f$ is a $\modal$-fibration, then $f$ is a $\modal$-fibration.
  \end{enumerate}
  \end{thm}
  \begin{proof}
    We apply Lemma \ref{lem:composing.cartesian.squares} to the squares
    \[
\begin{tikzcd}
	X & {Y } & Z \\
	{\modal X} & {\modal Y} & {\modal Z}
	\arrow[from=1-1, to=2-1]
	\arrow["{\modal f}"', from=2-1, to=2-2]
	\arrow["{\modal g}"', from=2-2, to=2-3]
	\arrow["f", from=1-1, to=1-2]
	\arrow[from=1-2, to=2-2]
	\arrow["g", from=1-2, to=1-3]
	\arrow[from=1-3, to=2-3]
\end{tikzcd}
    \]
    For the third part, remember that $g$ is $\modal$-{\'e}tale precisely when its naturality square is a pullback.
  \end{proof}

We now have the tools to characterize ${\modal}$-fibrations in another way. A modality is called \emph{lex} if it preserves all pullbacks. Not all modalities are lex; for example, the truncation modalities are not, and nor is $\shape$. The ${\modal}$-fibrations are precisely the maps along which ${\modal}$ is lex. That is, ${\modal}$ preserves all pullbacks of a map $f$ if and only if that map is a ${\modal}$-fibration.
\begin{thm}\label{thm:fibration.iff.all.pullbacks.preserved}
 A map $f : X \to Y$ is a ${{\modal}}$-fibration if and only if ${{\modal}}$ preserves every pullback of it in the sense that whenever the square on the left is a pullback, so is the square on the right.
\[
\begin{tikzcd}
A \arrow[d, "g"'] \arrow[r, "x"] & X \arrow[d, "f"] \\
B \arrow[r, "y"'] & Y
\end{tikzcd}
\quad\quad\quad\quad
\begin{tikzcd}
{{\modal}}A \arrow[d, "{{\modal}}g"'] \arrow[r, "{{\modal}}x"] & {{\modal}}X \arrow[d, "{{\modal}}f"] \\
{{\modal}}B \arrow[r, "{{\modal}}y"'] & {{\modal}}Y
\end{tikzcd}
\]
\end{thm}

\begin{rmk}
For the case of $\shape$, Theorem \ref{thm:fibration.iff.all.pullbacks.preserved} gives us a sufficient condition for a pullback to be a \emph{homotopy pullback} (that is, a pullback on homotopy types): if one of the legs is a $\shape$-fibration, then the pullback is a homotopy pullback.
\end{rmk}

\begin{proof}
If ${{\modal}}$ preserves all pullbacks of $f$, then by taking $B \equiv \ast$, we see that ${{\modal}}$ preserves all fibers of $f$ which by definition makes it a ${{\modal}}$-fibration.

    On the other hand, suppose that $f$ is a ${{\modal}}$-fibration and that the
    square on the left above is a pullback. Then the connecting map $\alpha :
    \fib_g(a) \to \fib_f(y a)$ is an equivalence for all $a : A$. Furthermore,
    $g$ is also a ${{\modal}}$-fibration by Corollary \ref{cor:fibrations.closed.under.pullback} and therefore the maps $\gamma_f : {{\modal}}\fib_f(y a) \to \fib_{{{\modal}}f}((y a)^{{\modal}})$ and $\gamma_g : {{\modal}}\fib_g(a) \to \fib_{{{\modal}}g}(a^{{\modal}})$ are equivalences for all $a : A$. These maps fit together into a commuting square:
\[
\begin{tikzcd}
{{\modal}}\fib_g(a) \arrow[r, "{\modal}\alpha"] \arrow[d, "\gamma_g"'] & {{\modal}}\fib_f (y a) \arrow[d, "\gamma_f"] \\
\fib_{{{\modal}}g}(a^{{\modal}}) \arrow[r] & \fib_{{{\modal}}f}((y a)^{{\modal}})
\end{tikzcd}
\]
Since the sides and top are equivalences, the bottom is also an equivalence.

Now, in order to show that the square on the right is a pullback, we need for the induced map $\zeta \colon \fib_{{{\modal}}g}(u) \to \fib_{{{\modal}}f}({{\modal}}y(u))$ to be an equivalence for all $u : {{\modal}}B$. But we have only shown it for $u \equiv a^{{\modal}}$, since ${{\modal}}y(a^{{\modal}}) = (y a)^{{\modal}}$ by naturality. Luckily, as both $\fib_{{{\modal}}g}(u)$ and $\fib_{{{\modal}}f}({{\modal}}y(u))$ are ${{\modal}}$-modal, $\type{isEquiv}(\zeta)$ is also ${{\modal}}$-modal for all $u : {{\modal}}B$. We may therefore assume that $u \equiv a^{{\modal}}$ by ${{\modal}}$-induction.
\end{proof}

As a corollary of this, we can prove a partial stability of the
${\modal}$-equivalence/${\modal}$-{\'e}tale factorization system. A factorization
system is stable if the left class is stable under pullback.

\begin{rmk}\label{rmk:modal.equiv.not.stable.under.pullback}
The class of ${\modal}$-equivalences is not stable under pullback in general. For example, consider the following pullback
\[
\begin{tikzcd}
\emptyset \arrow[r] \arrow[d] & \ast \arrow[d, "1"] \\
\ast \arrow[r, "0"'] & \Rb
\end{tikzcd}
\]
Though the bottom map is a $\shape$-equivalence since $\Rb$ is homotopically
contractible, the top map is not a $\shape$-equivalence.
\end{rmk}

On the other hand, ${\modal}$-equivalences are preserved by pullback along ${\modal}$-fibrations.
\begin{cor}
Suppose that the following square is a pullback. If $f$ is a ${\modal}$-fibration and $y$ a ${\modal}$-equivalence, then $x$ is a ${\modal}$-equivalence.
\[
\begin{tikzcd}
A \arrow[d, "g"'] \arrow[r, "x"] & X \arrow[d, "f"] \\
B \arrow[r, "y"'] & Y
\end{tikzcd}
\]

\end{cor}
\begin{proof}
Since $f$ is a ${\modal}$-fibration, the square
\[
\begin{tikzcd}
{{\modal}}A \arrow[d, "{{\modal}}g"'] \arrow[r, "{{\modal}}x"] & {{\modal}}X \arrow[d, "{{\modal}}f"] \\
{{\modal}}B \arrow[r, "{{\modal}}y"'] & {{\modal}}Y
\end{tikzcd}
\]
is also a pullback. But ${\modal} y$ is an equivalence by hypothesis, and therefore so is ${\modal} x$.
\end{proof}

All of this pullback preserving lets us add a few more conditions to the long list of equivalent conditions for lexness in Theorem 3.1 of \cite{RSS}.
\begin{prop}
The following are equivalent:
\begin{enumerate}
    \item The modality ${\modal}$ is lex.
    \item Every map is a ${\modal}$-fibration.
    \item If every map $f_i : A_i \to B_i$ is a ${\modal}$-fibration in a family of maps $f$, then the total map $\term{tot}(f) : \dsum{i : I} A_i \to \dsum{i : I}{B_i}$ is a ${\modal}$-fibration.
    \item For any map $f : X \to Y$, the connecting map $\term{tot}(\gamma) : \dsum{y : Y} {\modal} \fib_f(y) \to \dsum{y : Y}{\fib_{{\modal} f}(y^{{\modal}})}$ between factorizations of $f$ is a ${\modal}$-fibration.
    \item The universal map $\Type_{\ast} \to \Type$ is a ${\modal}$-fibration.
\end{enumerate}
\end{prop}
\begin{proof}
Conditions $1$ and $2$ are equivalent by the characterization of
${\modal}$-fibrations in terms of pullback preservation, and condition 2 trivially
implies conditions 3, 4, and 5. Every map between ${\modal}$-modal types is
${\modal}$-{\'e}tale since for ${\modal}$-modal types the modal units are
equivalences. Therefore, the connecting map $\gamma : {\modal} \fib_f(y) \to
\fib_{{\modal} f}(y^{{\modal}})$ is ${\modal}$-{\'e}tale and in particular a
${\modal}$-fibration for any map $f : X \to Y$ and $y : Y$. This means that
condition $3$ implies condition $4$. On the other hand, since ${\modal}$-fibrations
are closed under composition, if $\term{tot}(\gamma)$ is a ${\modal}$-fibration
then the ${\modal}$-modal factor of any map $f : X \to Y$ is a ${\modal}$-fibration,
as it is the composite of $\term{tot}(\gamma)$ and the ${\modal}$-{\'e}tale factor
of $f$. Therefore, by Lemma \ref{lem:fibration.iff.factor.etale}, $f$ is a
${\modal}$-fibration, so that condition 4 implies condition 2.

Finally, the last condition implies the second since ${\modal}$-fibrations
are closed under pullback.
\end{proof}

All objects are ``fibrant'' with respect to ${\modal}$-fibrations in the sense that the terminal map is always a ${\modal}$-fibration. We can say something more --- every projection map $\fst : A \times B \to A$ is a ${\modal}$-fibration.
\begin{lem}
For any types $A$ and $B$, the projection map $\fst : A \times B \to A$ is a ${\modal}$-fibration.
\end{lem}
\begin{proof}
This follows directly from the fact that ${\modal}$ preserves products. The map $(-)^{{\modal}} \times (-)^{{\modal}} : A \times B \to {\modal} A \times {\modal} B$ is a ${\modal}$-unit by Lemma 1.27 of \cite{RSS}, and so for any $a : A$ we get a map of fiber sequences:
\[
\begin{tikzcd}[column sep = large]
B \arrow[d] \arrow[r, "(-)^{{\modal}}"] & {\modal} B \arrow[d]\\
A \times B \arrow[d, "\fst"'] \arrow[r, "(-)^{{\modal}}\times(-)^{{\modal}}"] & {\modal} A \times {\modal} B \arrow[d, "\fst"] \\
A \arrow[r, "(-)^{{\modal}}"'] & {\modal} A
\end{tikzcd}
\]
where the bottom square is a ${\modal}$-naturality square. The induced map $\gamma : \modal \fib_{\fst}(a) \to \fib_{{\modal} \fst}(a^{{\modal}})$ is therefore equal to the identity map of ${\modal} B$, and so is an equivalence.
\end{proof}

A map $f : X \to Y$ is equal to a projection $\fst : Y \times Z \to Y$ if and only if $\fib_f : Y \to \Type$ is constant, that is, if it factors through the point.
\[
\begin{tikzcd}
Y \arrow[d] \arrow[r, "\fib_f"] & \Type \\
\ast \arrow[ur, dashed, "Z"'] &
\end{tikzcd}
\]
We have just shown that such maps are ${\modal}$-fibrations, but we can do better. We can show that a map is a ${\modal}$-fibration if and only if it has \emph{${\modal}$-locally constant ${\modal}$-fibers} in the sense made precise in the upcoming Theorem \ref{thm:fibration.iff.locallyconstant.modal.fibers}. First, we prove a similar characterization of ${\modal}$-{\'e}tale maps. This is the \emph{modal descent} theorem of \cite{Rijke-Cherubini:Modal.Descent}.

\begin{lem}\label{lem:etale.iff.locallyconstant.modal}
Let $E : Y \to \Type_{{\modal}}$ be a family of modal types. Then $E$ factors through the modal unit of $Y$ if and only if $\fst : \dsum{y : Y} Ey \to Y$ is ${{\modal}}$-{\'e}tale. In particular, the type of such factorizations is a proposition.
\end{lem}
\begin{proof}
If $\fst$ is ${{\modal}}$-{\'e}tale, then $\gamma : Ey \to \fib_{{{\modal}}\fst}(y^{{\modal}})$ is an equivalence; therefore, $\fib_{{{\modal}}\fst} : {{\modal}}Y \to \Type_{{\modal}}$ is such a factorization.

On the other hand, suppose that $\tilde{E} : {{\modal}}Y \to \Type_{{\modal}}$ with $w : \dprod{y : Y} (Ey \simeq \tilde{E}y^{{\modal}})$ is a factorization. Then the square
\[
\begin{tikzcd}
\dsum{y : Y} Ey \arrow[d, "\fst"'] \arrow[r, "\term{tot}(w)"] & \dsum{u : {{\modal}}Y} \tilde{E}u \arrow[d, "\fst"] \\
Y \arrow[r] & {{\modal}}Y
\end{tikzcd}
\]
is a pullback. Since the unit $Y \to {\modal} Y$ is ${\modal}$-connected and
${\modal}$-connected maps are closed under pullback, $\term{tot}(w)$ is
${\modal}$-connected. As $\dsum{u : {\modal} Y} \tilde{E}u$ is a sum of modal types
over a modal type, it is modal, and therefore $\term{tot}(w)$ is a ${\modal}$-unit
and this square is a ${\modal}$-naturality square. But then $\fst : \dsum{y : Y} Ey
\to Y$ is ${\modal}$-{\'e}tale since its ${\modal}$-naturality square is a pullback.

To show that the type of such factorizations is a proposition, we just need to
show that any factorization equals $(\fib_{{\modal}\fst},\, \gamma)$. This follows
immediately from the uniqueness of ${\modal}$-units.
\end{proof}

As a corollary, we can characterize the ${\modal}$-{\'e}tale maps into a type $Y$.
\begin{cor}\label{cor:characterizing.etale.maps}
  For any type $Y$, the type $$\mbox{{\'E}t}_{{\modal}}(Y) :\equiv \dsum{X :
    \Type}\dsum{f : X \to Y}\mbox{is${\modal}${\'e}tale}(f)$$
  is equivalent to the type ${\modal} Y \to \Type_{ {\modal} }$ of families of modal types
  varying over ${\modal} Y$.
\end{cor}
\begin{proof}
  Consider the following equivalence:
  \begin{align*}
    \mbox{{\'E}t}_{{\modal}}(Y) &:\equiv \dsum{X : \Type}\dsum{f : X \to Y}\mbox{is${\modal}${\'e}tale}(f) \\
                             &\simeq \dsum{X : \Type}  \dsum{f : X \to Y} \dsum{\tilde{E} : {\modal} Y \to \Type_{{\modal}}} \fib_f = \tilde{E} \circ (-)^{{\modal}} \\
    &\simeq \dsum{E : Y \to \Type_{{\modal}}} \dsum{\tilde{E} : {\modal} Y \to \Type_{{\modal}}}  (E = \tilde{E} \circ (-)^{{\modal}}) \\
 &\simeq {\modal} Y \to \Type_{{\modal}} \qedhere
  \end{align*}
\end{proof}

We may now prove the main theorem of this section, characterizing $\modal$-fibrations as those maps with $\modal$-locally constant $\modal$-fibers.
\begin{thm}\label{thm:fibration.iff.locallyconstant.modal.fibers}
Let $E : Y \to \Type$ be a family of types. Then $\fst : \dsum{y : Y} Ey \to Y$ is a ${{\modal}}$-fibration if and only if there is a type family $\tilde{E} : {{\modal}}Y \to \Type_{{\modal}}$ making the following square commute:
\[
\begin{tikzcd}
Y \arrow[d] \arrow[r, "E"] & \Type \arrow[d, "{\modal}"] \\
{{\modal}}Y \arrow[r, dashed, "\tilde{E}"'] & \Type_{{\modal}}
\end{tikzcd}
\]
\end{thm}
\begin{rmk}
In the case of the $\shape$ modality, Theorem \ref{thm:fibration.iff.locallyconstant.modal.fibers} can be understood as characterizing the $\shape$-fibrations as those maps whose fibers form a \emph{local system} on their codomain. The factorization $\tilde{E} : \shape Y \to \Type_{{\shape}}$ of $\shape E : Y \to \Type_{\shape}$ shows that the homotopy types of the fibers $Ey$ are \emph{locally constant} in $y$. Moreover, the usual transport of identifications in $\shape Y$ gives rise to a \emph{monodromy} action of the homotopy type $\shape Y$ on the homotopy types $\shape E y$ of the fibers $E y$.
  \end{rmk}
\begin{proof}
By Lemma \ref{lem:fibration.iff.factor.etale}, $\fst$ is a fibration if and only if its modal factor $\Ra(\fst) : \dsum{y : Y}{{{\modal}}(Ey)} \to Y$ is ${{\modal}}$-{\'etale}. By Lemma \ref{lem:etale.iff.locallyconstant.modal}, $\Ra(\fst)$ is ${{\modal}}$-{\'etale} if and only if ${{\modal}}E : Y \to \Type_{{\modal}}$ factors through ${{\modal}}Y$. But this is exactly what we are asking for!
\end{proof}

What is a $\trunc{-}_n$-fibration? A map is a $\trunc{-}_n$-equivalence exactly
when it induces an equivalences on the homotopy groups $\pi_k$ for $0 \leq k
\leq n$ (see Theorem 8.8.3 of \cite{HoTTBook}), and is $\trunc{-}_n$-connected when it furthermore induces a surjection
on $\pi_{n+1}$ (see Corollary 8.8.6 of \cite{HoTTBook}). Since a map is a $\trunc{-}_n$-fibration if and only if its
$\trunc{-}_n$-equivalence factor is $\trunc{-}_n$-connected, we might expect
that a map is a $\trunc{-}_n$-fibration if it induces a surjection on
$\pi_{n+1}$. We can prove this naive conjecture by giving one more equivalent
characterization of ${\modal}$-fibrations --- this time with a small caveat.

We first need an elementary lemma concerning fibers.
\begin{lem}\label{lem:3x3fiber.lemma}
  Consider a square
  \[
\begin{tikzcd}
	A & B \\
	C & D
	\arrow["f"', from=1-1, to=2-1]
	\arrow["g", from=1-1, to=1-2]
	\arrow["h", from=1-2, to=2-2]
	\arrow["k"', from=2-1, to=2-2]
\end{tikzcd}
  \]
  commuting via $S : \dprod{x : A} (k(f(x)) = h(g(x)))$. Let $a : A$, and define $F : \fib_{g}(g a) \to \fib_{k}(k f a)$ by
  \[
F(x : A,\, p : gx = g a) :\equiv (f x, S(x) \cdot h_{\ast}p \cdot S(a)\inv).
  \]
  For $(c, q) : \fib_{k}(k f a)$, define $G : \fib_{f}(c) \to \fib_{h}(k f a)$ by
  \[
G(x : A, w : f x = c) :\equiv (g x, S(x)\inv \cdot k_{\ast} w \cdot q).
  \]
  Then we have an equivalence $\fib_{F}(c, q) = \fib_{G}(g a, S(a)\inv)$ giving a (judgementally) commuting square
  \[
\begin{tikzcd}
	{\fib_F(c, q)\,\,(= \fib_G(ga, S(a)\inv)} & {\fib_f(c)} \\
	{\fib_{g}(g a)} & A
	\arrow[from=1-1, to=2-1]
	\arrow[from=2-1, to=2-2]
	\arrow[from=1-1, to=1-2]
	\arrow[from=1-2, to=2-2]
\end{tikzcd}
  \]
  \end{lem}
  \begin{proof}
We find the equivalence as the following composite:
\begin{align*}
  \fib_{F}(c, q) &:\equiv \dsum{(x, p) : \fib_{g}(g a)} (F(x, p) = (c, q)) \\
&= \dsum{x : A} \dsum{p : gx = ga} ((f x, S(x) \cdot h_{\ast}p \cdot S(a)\inv) = (c, q)) \\
  &= \dsum{x : A} \dsum{p : gx = ga} \dsum{w : fx = c} (k_{\ast}w\inv \cdot S(x) \cdot h_{\ast}p \cdot S(a)\inv = q)\\
                 &= \dsum{x : A} \dsum{w : fx = c} \dsum{p : g x = g a} (h_{\ast}p\inv \cdot S(x) \cdot k_{\ast}w \cdot q = S(a)\inv)\\
                   &= \dsum{x : A} \dsum{w : fx = c} (G(x, w) = (ga, S(a)\inv))\\
                     &= \fib_{G}(ga, S(a)\inv).
\end{align*}
Note that throughout this equivalence, $x : A$ is not affected by the equivalences. Therefore, we end up with the judgementally commuting square as desired.
    \end{proof}

\begin{thm}\label{lem:fibration.iff.map.on.fibers.connected}
Let $f : X \to Y$.
\begin{enumerate}
        \item If $f$ is a $\modal$-fibration, then for all $x : X$ the induced map $\fib_{(-)^{\modal}}(x^{\modal}) \to \fib_{(-)^{\modal}}((fx)^{\modal})$ is $\modal$-connected.
        \item If the modal unit $(-)^{\modal} : X \to \modal X$ is surjective, and for all $x : X$ the induced map $\fib_{(-)^{\modal}}(x^{\modal}) \to \fib_{(-)^{\modal}}((fx)^{\modal})$ is $\modal$-connected, then $f$ is a $\modal$-fibration.
\end{enumerate}
\end{thm}
\begin{proof}
  First, suppose that $f$ a $\modal$-fibration, and let $x : X$ seeking to show that the induced map $\fib_{(-)^{\modal}}(x^{\modal}) \to \fib_{(-)^{\modal}}((fx)^{\modal})$ is $\modal$-connected. By Lemma \ref{lem:3x3fiber.lemma}, the fiber of the induced map over $(y, p) : \fib_{(-)^{\modal}}((fx)^{\modal})$ is equivalent to the fiber of $\delta : \fib_{f}(y) \to \fib_{\modal f}(y^{\modal})$ over $(x^{\modal},  S(x)\inv)$ where $S : \dprod{x : X} (f x)^{\modal} = \modal f(x^{\modal})$ is witness to the commutativity of the naturality square. Since $f$ is a $\modal$-fibration, this $\delta$ is a $\modal$-equivalence; but it is a $\modal$-equivalence landing in a modal type, and is therefore a $\modal$-unit, which is to say it is $\modal$-connected.

Conversely, suppose that the modal unit $(-)^{\modal} : X \to \modal X$ is surjective. We aim to show that $f : X \to Y$ is a $\modal$-fibration, so it suffices to prove that the maps $\delta : \fib_{f}(y) \to \fib_{\modal f} (y^{\modal})$ are $\modal$-connected for all $y : Y$. So, suppose we have $(u, p) : \fib_{\modal f}(y^{\modal})$, seeking to show that $\fib_{\delta}(u, p)$ is $\modal$-connected. By the surjectivity of $(-)^{\modal} : X \to \modal X$, we may assume $u$ is of the form $x^{\modal}$. Then Lemma \ref{lem:3x3fiber.lemma} tells us that $\fib_{\delta}(x^{\modal}, p)$ is equivalent to the fiber of the induced map $\fib_{(-)^{\modal}}(x^{\modal}) \to \fib_{(-)^{\modal}}((fx)^{\modal})$ over $(fx, S(x))$. But by hypothesis, this fiber was $\modal$-connected.
  \end{proof}

  \begin{rmk}
The condition that $(-)^{\modal} : X \to \modal X$ be surjective is often trivially satisfied. For many modalities --- the $n$-truncation modalities and the shape modality included --- \emph{all} modal units are surjective. In this case, Theorem \ref{lem:fibration.iff.map.on.fibers.connected} characterizes the $\modal$-fibrations with no caveats. We might refer to modalities whose units are surjective as \emph{global} modalities; they are counterposed to topological modalities, which are given by a nullification at a family of propositions, since any global topological modality is trivial. More specifically, any global modality is cotopological in the sense of Theorem 3.22 of \cite{RSS}.
    \end{rmk}

\begin{cor}\label{cor:truncation.fibrations}
  A map $f : X \to Y$ is a $\trunc{-}_n$-fibration if and only if for all $y :
  Y$ and $(x, p) : \fib_f(y)$, the induced map $\pi_{n+1}(X, x) \to \pi_{n+1}(Y,
  y)$ is surjective.
\end{cor}
\begin{proof}
  By Theorem \ref{lem:fibration.iff.map.on.fibers.connected}, $f$ is a
  $\trunc{-}_n$-fibration if and only if the induced map $\fib_{|-|_n}(x) \to
  \fib_{|-|_n}(y)$ is $\trunc{-}_n$-connected. As the fibers of $\trunc{-}_n$-units, $\fib_{|-|_n}(x)$ and
  $\fib_{|-|_n}(y)$ are $\trunc{-}_n$-connected, so the induced map is
  $\trunc{-}_{n}$-connected if and only if the induced map $\pi_{n+1}(\fib_{|-|_n}(x), (x,
  \refl)) \to \pi_{n+1}(\fib_{|-|_n}(y), (y, \refl))$ is a surjection. But this
  map is equivalent to the induced map $\pi_{n+1}(X, x) \to \pi_{n+1}(Y, y)$.
\end{proof}

Before moving on, let's briefly consider a pair of modalities ${\modal} \leq\, {\modaltwo}$,
where every ${\modal}$-modal type is ${\modaltwo}$-modal. For example, $\trunc{-}_n \leq
\trunc{-}_{n+1}$. In particular, ${\modal} X$ is
${\modaltwo}$-modal, and so the unit $(-)^{{\modal}} : X \to {\modal} X$ factors uniquely
through $(-)^{{\modaltwo}} : X \to {\modaltwo} X$, giving us a commuting diagram:
\[
  \begin{tikzcd}
    X \arrow[r,"(-)^{{\modaltwo}}" ] \arrow[dr,"(-)^{{\modal}}"'] & {\modaltwo} X \arrow[d, dashed,"c" ] \\
     & {\modal} X
  \end{tikzcd}
\]
\begin{lem}\label{lem:submodality.connecting.unit}
Suppose that every ${\modal}$-modal type is ${\modaltwo}$-modal. Then the connecting map $c :
{\modaltwo} X \to {\modal} X$ is a ${\modal}$-unit. As a corollary, for any $f : X \to Y$, we get a ${\modal}$-naturality square
\[
  \begin{tikzcd}
    {\modaltwo} X \arrow[d,"{\modaltwo} f"'] \arrow[r] & {\modal} X \arrow[d,"{\modal} f"] \\
    {\modaltwo} Y \arrow[r] & {\modal} Y
  \end{tikzcd}
\]
\end{lem}
\begin{proof}
Let $Z$ be a ${\modal}$-modal type. It is therefore also ${\modaltwo}$-modal. Precomposing by
the above commutative triangle gives us a commutative diagram:
\[
  \begin{tikzcd}
    ( X \to Z ) \arrow[r, leftarrow,"\sim" ] \arrow[dr,leftarrow, "\sim"' ] & ( {\modaltwo} X \to Z ) \arrow[d, leftarrow] \\
     & (  {\modal}  X \to Z )
  \end{tikzcd}
\]
Because $Z$ is both ${\modal}$-modal and ${\modaltwo}$-modal, the two horizontal maps are
equivalences, and therefore the vertical map is an equivalence, as desired.
\end{proof}

We
aim to demonstrate the following relations between the different kinds of maps
associated to these modalities.
\begin{thm}\label{thm:comparing.modalities}
  Suppose that every ${\modal}$-modal type is ${\modaltwo}$-modal, and that $f : X \to Y$.
  Then:
\begin{enumerate}
\item If $f$ is ${\modal}$-modal, then it is ${\modaltwo}$-modal.
\item If $f$ is ${\modal}$-{\'e}tale, then it is ${\modaltwo}$-{\'e}tale.
\item If $f$ is a ${\modaltwo}$-equivalence, then it is a ${\modal}$-equivalence.
\item If $f$ is ${\modaltwo}$-connected, then it is ${\modal}$-connected.
  \item If $f$ is a ${\modaltwo}$-fibration and ${\modaltwo} f$ is a ${\modal}$-fibration, then $f$ is a ${\modal}$-fibration.
\end{enumerate}
\end{thm}
\begin{proof}[Proof of Theorem \ref{thm:comparing.modalities}]
  $\quad$

  \begin{enumerate}
    \item If $f$ is ${\modal}$-modal, then its fibers are ${\modal}$-modal and so by
      hypothesis ${\modaltwo}$-modal, so that $f$ is ${\modaltwo}$-modal.
    \item If $f$ is ${\modal}$-{{\'e}tale}, then by Lemma
      \ref{lem:etale.iff.locallyconstant.modal}, $\fib_f$ factors through ${\modal}
      X$ as $E : {\modal} X \to \Type$. But then $E \circ c : {\modaltwo} X \to \Type$ is a
      factorization of $\fib_f$ through ${\modaltwo} X$, so that $f$ is ${\modaltwo}$-{{\'e}tale}.
    \item If $f$ is a ${\modaltwo}$-equivalence, then ${\modaltwo} f$ is an equivalence. But then
      since ${\modal} {\modaltwo} f$ is equivalent to ${\modal} f$ by Lemma
      \ref{lem:submodality.connecting.unit}, ${\modal} f$ is an equivalence.
    \item If $f$ is ${\modaltwo}$-connected, then ${\modaltwo}\fib_f(y)$ is contractible for all $y
      : Y$. But then ${\modal} \fib_f(y) = {\modal} {\modaltwo}\fib_f(y)$ is contractible for
      all $y : Y$, so $f$ is ${\modal}$-connected.
    \item Consider the following diagram.
      \[
        \begin{tikzcd}
          Y \arrow[r,"\fib_f" ] \arrow[d] & \Type \arrow[d,"{\modaltwo}" ]\\
          {\modaltwo} Y \arrow[r,"\fib_{{\modaltwo} f}"] \arrow[d] & \Type_{{\modaltwo}} \arrow[d,"{\modal}" ]\\
          {\modal} Y \arrow[r,"\fib_{{\modal} {\modaltwo} f}"] & \Type_{{\modal}}
        \end{tikzcd}
      \]
      If $f$ is a ${\modaltwo}$-fibration then the upper square commutes, and if ${\modaltwo}f$ is a
      ${\modal}$-fibration then the lower square commutes. If the outer square
      commutes, then $\fib_f$ factors through ${\modal} Y$, and so is a
      ${\modal}$-fibration. \qedhere
  \end{enumerate}
\end{proof}

\section{A Brief Review of Cohesive HoTT}\label{sec:cohesion.review}
In this section, we review Mike Shulman's Real Cohesive Homotopy Type Theory (as
found in \cite{RealCohesion}). The \emph{shape} modality $\shape$ which sends a
type to its homotopy type is defined in the context of Real Cohesive HoTT. It is
the interplay of this modality with the \emph{comodality} $\flat$ that defines
real cohesion, and that we will exploit to give a trick for showing that a map
is a $\shape$-fibration.

For the reader who isn't too familiar with real cohesion and doesn't feel like
getting too familiar with it, worry not. The details in this section revolve
around the notion of \emph{crisp} objects, which will be explained below. But
every object (type or element) which appears in the empty context --- that is to
say, with no free variables in its definition --- is crisp. Therefore, if you
need a heuristic for understanding what it means to, say, have a crisp type $Z
:: \Type$, just imagine that this means that $Z$ has no free variables in its
definition. For example, $\Nb$, $\Zb$, $\Rb$, and $\Type$ are all crisp types,
while $0 : \Nb$, $\pi : \Rb$, and $\lam{x} x^2 + 2 : \Rb \to \Rb$ are all crisp
elements since they have no free variables. Furthermore, any natural number may
be assumed to be crisp, so that types like $\Rb^n$ may be taken as crisp even though they
involve a free variable $n : \Nb$.

In type theory, if you can argue that for all $x : X$, there is an $f(x) : Y$, then you have given a function $f : X \to Y$ in the process. In Shulman's Real Cohesive HoTT, all functions will be \emph{continuous} in a topological sense. So, saying that for $x : X$ we have a $f(x) : Y$ means that $f(x)$ must depend \emph{continuously} on $x$. But not all dependencies are continuous. What if we want to express a discontinuous dependence?

To address this concern, Shulman introduces the notion of a ``crisp variable''
$$a :: A$$
to express a discontinuous dependence. Hypothesizing $a :: A$ means that we can use $a$ in a discontinuous manner; one way this is realized is in the crisp Law of Excluded middle.

\begin{axiom}[Crisp excluded middle]
For any crisp $P :: \Prop$, we have $P \vee \neg P$.
\end{axiom}

This axiom lets us use case analysis when assuming a crisp element of a set, even if the set has a native topology that wouldn't admit case analysis constructively (such as the Dedekind real numbers $\Rb$, which cannot constructively be separated into two disjoint parts).

Any variable appearing in the type of a crisp variable must also be crisp, and a
crisp variable may only be substituted by expressions that \emph{only} involve
crisp variables. When all the variables in an expression are crisp, we say that
that expression is crisp; so, we may only substitute crisp expressions in for
crisp variables. Constants --- like $0 : \Nb$ or $\Nb : \Type$ --- appearing in an empty context are therefore always crisp. This means that one cannot give a closed form example of a term which is \emph{not} crisp; all terms with no free variables are crisp. For emphasis, we will say that a term which is not crisp is \emph{cohesive}. The rules for crisp type theory can be found in Section 2 of \cite{RealCohesion}.

One way to think of the difference between a cohesive dependence --- for all $x :
X$, $f(x) : Y$ --- and a crisp dependence --- for all $x :: X$, $f(x) : Y$ --- is
that the former expresses that $f(x)$ depends on a generic $x : X$, whereas in
the latter we are saying that for \emph{each} individual $x$, there is an
$f(x)$.\footnote{In particular, by the crisp excluded middle axiom, we may deal
  with each $x :: X$ on a case by case basis.}

Given a \emph{crisp} type $X$, we can remove its spatial structure to get a
type $\flat X$. If $X$ is a set, $\flat X$ can be thought of as its set
of points.\footnote{ This intuition really only works for sets, since if $G$ is
  a group then $\flat \BB G$ behaves like the moduli stack of principal $G$-bundles with flat connection, and not ``the type of points of $\BB G$''. } The rules for $\flat$ can be found in Section 4 of \cite{RealCohesion}. They may be summed up by saying that $\flat X$ is inductively generated by elements of the form $x^{\flat}$ for \emph{crisp} $x :: X$. In particular, whenever we have a type family $C : \flat X \to \Type$, an $x : \flat X$, and an element $f(u) : C(u^{\flat})$ depending on a \emph{crisp} $u :: X$, we get an element
$$( \mbox{let $u^{\flat} := x$ in $f(u)$} ) : C(x)$$
and if $x \equiv v^{\flat}$, then $(\mbox{let $u^{\flat} := x$ in $f(u)$})
\equiv f(v)$. This allows us to think of $\flat X$ as ``the type of \emph{crisp}
points of $X$''.

We have an inclusion $(-)_{\flat} : \flat X \to X$ given by $x_{\flat} :\equiv
\mbox{let $u^{\flat} := x$ in $u$}$. Since we are thinking of a dependence on a
crisp variable as a \emph{discontinuous} dependence, if this map $(-)_{\flat} :
\flat X \to X$
is an equivalence then every \emph{discontinuous} dependence on $x :: X$
underlies a \emph{continuous} dependence on $x$. This leads us to the following defintion:
\begin{defn}
A crisp type $X :: \Type$ is \emph{crisply discrete} if the counit $(-)_{\flat} : \flat X \to X$ is an equivalence.\footnote{See Remark 6.13 of \cite{RealCohesion} for a discussion on some of the subtleties in the notion of crisp discreteness.}
\end{defn}

We would like our formal notion of continuity coming from
crisp types to match our topological notion of continuity as measured by
continuous paths. We have a notion of discreteness coming from crisp variables
--- \emph{crisply discrete} --- but we also need a topological notion of
discreteness.
\begin{defn}\label{defn:discrete.real.cohesion}
A type $X$ is \emph{discrete} if every path in it is constant in the sense that
the inclusion of constant paths $X \to (\Rb \to X)$ is an equivalence.
\end{defn}

\begin{rmk}
The real numbers $\Rb$ in Definition \ref{defn:discrete.real.cohesion} --- and throughout this paper --- are the \emph{Dedekind} real numbers and not the Cauchy real numbers. It can be proven in real cohesion (with a form of the axiom of choice) that the Cauchy real numbers are discrete, and that indeed they are equivalent to $\flat \Rb$ --- see Corollary 8.28 of \cite{RealCohesion}.
  \end{rmk}

Note that we can form the proposition ``is discrete'' for any
type, while we can only form the proposition ``is crisply discrete'' for crisp
types, since to form $\flat X$, $X$ must be crisp. The main axiom of real cohesion, which ties the liminal sort of topology implied by the use of crisp variables to the concrete topology of the real numbers, is that for crisp types being discrete and being crisply discrete coincide.
\begin{axiom}[$\Rb\flat$]
  A crisp type $X :: \Type$ is crisply discrete if and only if it is discrete.
\end{axiom}

We can now define the shape modality as a localization.
\begin{defn}
The \emph{shape} or \emph{homotopy type} $\shape X$ of a type $X$ is defined to be
the localization of $X$ at the type of Dedekind real numbers $\Rb$ (see
Definition 9.6 of \cite{RealCohesion}). By construction, a type is $\shape$-modal
if and only if it is discrete.
\end{defn}

Since $\shape$ is given by localization at a small type,\footnote{Assuming
  propositional resizing, $\Rb$ is as small as $\Nb$; without propositional
  resizing, $\Rb$ has the size of the universe of $\Nb$. We will assume
  propositional resizing here, as is common in homotopy type theory and valid in
any $\infty$-topos.} it is accessible in the sense of \cite{RSS}. Therefore, by Lemma 2.24 of \cite{RSS}, it may be extended canonically to any larger universe. For this reason, and because $\flat$ is universe polymorphic, we will elide the size issues in the use of $\shape$ and, for example, consider the type of discrete types $\Type_{\shape}$ to be $\shape$-separated.

In the upcoming sections, we will need not only the shape modality $\shape$, but
the \emph{$n$-truncated} shape modality $\shape_n$.
\begin{defn}
  Let $\shape_n$ be the modality whose modal types are discrete, $n$-truncated
  types. It can be constructed by localizing at the real line $\Rb$ and the
  homotopy $n$-sphere $S^n$.
\end{defn}

It may be tempting to define $\shape_n X$ as $\trunc{\shape X}_n$, but it is not
currently known whether $\trunc{D}_n$ of a discrete type $D$ is discrete; the
author suspects that it is not true in general. However, for crisp types, this
is true.
\begin{prop}\label{prop:truncation.shape}
  Let $X :: \Type$ be a crisp type. Then $\shape_n X = \trunc{\shape X}_n$.
\end{prop}
\begin{proof}
  Since $X$ is crisp, so is $\shape X$. Since $\shape X$ is crisp, $\trunc{\shape
    X}_n$ is crisply an $n$-type. Then, by Corollary 6.7 of \cite{RealCohesion},
  $\flat \trunc{ \shape X }_n = \trunc{\flat \shape X}_n$. But $\shape X$ is discrete, so by Axiom
  $\Rb\flat$, $\flat \shape X = \shape X$. Therefore, $\trunc{\shape X}_n$ is a
  discrete $n$-type
  and so the canonical map $\trunc{\shape X}_n \to \shape_n X$ is an equivalence.
\end{proof}

We can think of $\shape_n X$ as the ``fundamental $n$-groupoid'' of $X$. In
particular,
\begin{itemize}
\item $\shape_0 X$ is the set of connected components of $X$.
\item $\shape_1 X$ is the fundamental groupoid of $X$.
\end{itemize}

We can prove that $\shape_0 X$ is the set of connected components of $X$ in a
naive sense.
\begin{defn}
  Let $X$ be a type. A \emph{connected component} of $X$ is a subtype $C : X \to
  \Prop$ of $X$ which is
  \begin{enumerate}
  \item Inhabited: there is merely an $x : X$ such that $C(x)$.
    \item Connected: If $C \subseteq P \cup \neg P$, then $C \subseteq P$ or $C
      \subseteq \neg P$.\footnote{This expresses the connectivity of $C$ because
      it says that if $C$ is contained in a \emph{disjoint} union, it is
      contained wholly in one part.}
    \item Detachable: For any $x : X$, either $C(x)$ or $\neg
      C(x)$.\footnote{This says that $C$ is a \emph{component} of $X$ in the
        sense that
        $X$ is the disjoint union of $C$ and its complement.}
  \end{enumerate}
  We denote the set of connected components of $X$ by $\pi_0 X$.
\end{defn}

Connected components are quite rigid; if two connected components have non-empty
intersection, then they are equal.
\begin{lem}\label{lem:connected.component.equality}
  Suppose that $C$ and $D$ are connected components of $X$. Then $C = D$ if and
  only if $C \cap D$ is non-empty.
\end{lem}
\begin{proof}
If $C = D$, then $C \cap D$ is $C$ and so is inhabited.

  Since $D$ is detachable, we have that $X \subseteq D \cup \neg D$, and therefore $C \subseteq D \cup \neg D$. Now, $C$ is connected, so $C \subseteq D$ or $C \subseteq \neg D$; but it can't be the latter because then their intersection would be empty. So, $C \subseteq D$ and symmetrically $D \subseteq C$.
\end{proof}

Intuitively, $\shape_0 X$ should be the set of connected components of $X$
and $(-)^{\shape_0} : X \to \shape_0 X$ should send $x : X$ to the connected
component $x^{\shape_0}$ it is contained in. We can justify this intuition with the following
theorem.
\begin{lem}\label{lem:elements.of.shape.0.are.connected.components}
  Let $u : \shape_0 X$, and let $C_u : X \to \Prop$ be defined by
  $$C_u(x) :\equiv u = x^{\shape_0}$$
  Then $C_u$ is a connected component of $X$, giving us a map $C : \shape_0 X \to
  \pi_0 X$.
\end{lem}
\begin{proof}
  We need to prove that $C_u$ is inhabited, connected, and detachable.
  \begin{enumerate}
    \item $C_u$ is inhabited because $(-)^{\shape_0}$ is surjective (by the same
      proof as that of Corollary 9.12 of \cite{RealCohesion}).
    \item Suppose that $C_u \subseteq P \cup \neg P$. Consider the map $\chi : \dsum{x
        : X} C_u(x) \to \{0,\, 1\}$ sending $x$ to $0$ if $P(x)$ and $x$ to $1$
      if $\neg P(x)$. As $\{0,\, 1\}$ is a discrete set (by Theorems 6.19 and 6.21 of \cite{RealCohesion}, noting that $\{0, 1\} = \{0\} +
      \{1\}$), $\chi$ factors uniquely through $\shape_0(\dsum{x : X} C_u(x))$.
      But $\dsum{x : X} C_u(x) \equiv \fib_{(-)^{\shape_0}}$ is a fiber of a $\shape_0$-unit, and so is $\shape_0$-connected. Therefore $\chi$ is constant, and
      so either all $x$ in $C_u$ satisfy $P$, or they all satisfy $\neg P$.
    \item Since $\shape_0 X$ is a discrete set,
      it has decideable equality by Lemma 8.15 of \cite{RealCohesion}. So, for any $x : X$, either $u = x^{\shape_0}$
      or not. But that exactly means that $C_u(x)$ or not.\qedhere
  \end{enumerate}
\end{proof}

\begin{thm}
  Let $X$ be a type. Then the map $C : \shape_0 X \to \pi_0 X$ of Lemma
  \ref{lem:elements.of.shape.0.are.connected.components} is an equivalence.
\end{thm}
\begin{proof}
  We will show that the map $C$ is surjective and injective.
\begin{enumerate}
\item To show that $C$ is surjective, suppose that $U$ is a connected component
  of $X$, seeking to witness $\trunc{\fib_C(U)}$. Since we are seeking a
  proposition and $U$ is inhabited, we may assume that $x : X$ is in $U$. Then
  $x$ is in $C_{x^{\shape_0}} \cap U$, so that $C_{x^{\shape_0}} = U$ by Lemma \ref{lem:connected.component.equality}.
\item To show that $C$ is injective, suppose that $C_u = C_v$ seeking to show
  that $u = v$. If $C_u = C_v$, then $C_u \cap C_v = C_u$ is merely inhabited.
  Since we are seeking a proposition, let $x$ be an element in the intersection.
  But then $u = x^{\shape_0}$ and $v = x^{\shape_0}$, so $u = v$.  \qedhere
\end{enumerate}
\end{proof}

\begin{rmk}
  Though we have framed this paper as taking place in the setting of Real
  Cohesion, it will in fact mostly use the ``locally contractible'' part of the
  theory --- namely, crisp variables, the comodality $\flat$, the modality
  $\shape$, and the axiom relating them for crisp types. The only extra condition
  is that $\flat$ commute with propositional truncation, which, as proven in
  \cite{RealCohesion}, uses the codiscrete modality $\#$. It also follows from
  the fact (Proposition 8.8 of \cite{RealCohesion}) that propositions are
  discrete which only uses that $\shape$ is given by localization at a family of
  pointed types.

  In particular, Theorem \ref{thm:locally.crisply.discrete.BAut.discrete} replies only on crisp type theory, while Theorem \ref{thm:characterizing.shape.fibration} relies on the adjoint relationship of $\shape$ and $\flat$ (namely, that crisp types are $\shape$-modal if and only if they are $\flat$-comodal). Theorems \ref{thm:group.action.fibrations} and \ref{thm:shape.preserves.crisp.connectedness} relies only on Theorem \ref{thm:characterizing.shape.fibration}, and are therefore also valid in general cohesion. On the other hand, the specific examples in Sections \ref{sec:examples}, \ref{sec:higher.groups}, \ref{sec:connectedness} and \ref{sec:Covering.Theory} take place in real cohesion.

Therefore, the theory of $\shape$-fibrations and coverings in the coming sections should work equally well in other
settings that have an adjoint ${\modal} \dashv\, \comodal$ modality/comodality pair
implemented using crisp variables in which $\comodal$ preserves propositional
truncation. A likely example of such a situation would be the adjoint pair
$\mathfrak{I} \dashv \&$ between the crystaline modality $\mathfrak{I}$ which is
given by localizing at a family of infinitesimal types, and the infinitesimal
flat modality $\&$ which appears (in the language of $\infty$-toposes, rather
than type theory) in Schreiber's \cite{DiffCohomologyCohesiveTopos}. Since
$\mathfrak{I}$ is the localization at a family of pointed types,
propositions are crystaline and so $\&$ commutes with propositional truncation. In this setting, Theorem \ref{thm:characterizing.shape.fibration} would be used with Lemma \ref{lem:etale.iff.locallyconstant.modal} to show that the projections of certain bundles are $\mathfrak{I}$-{\'e}tale (that is, formally {\'e}tale or locally diffeomorphic).

The modality $\mathfrak{I}$ is left exact, and so every map is an
$\mathfrak{I}$-fibration. However, $\mathfrak{I}$-{\'e}tale maps include the
\emph{formally {\'e}tale} maps, or \emph{local diffeomorphisms}. So the
applications to covering theory of Section \ref{sec:Covering.Theory} can be
interpreted in this setting as well.
\end{rmk}
\section{Classifying Types of Discrete Structures are Discrete}\label{sec:discrete.classifying.types}

In this section, we will show that the classifying types of bundles of crisply discrete structures are themselves discrete. As a corollary, the fibers of such a bundle depend only on the homotopy type of the base space. We will use this fact to show that maps whose fibers have a merely constant homotopy type --- merely equivalent to some crisply discrete type --- are $\shape$-fibrations.

First, we need a good notion of ``type of discrete objects''. We will call these types \emph{locally discrete}.
\begin{defn}
A type $X$ is \emph{locally discrete} if it is $\shape$-separated, that is, for all $x,\, y : X$, $x = y$ is discrete. A crisp type $X$ is \emph{locally crisply discrete} if for all crisp $x,\, y :: X$, $x = y$ is crisply discrete; more explicitly, for all $x,\, y : \flat X$, $x_{\flat} = y_{\flat}$ is crisply discrete.
\end{defn}

\begin{rmk}
We can't explicitly quantify over crisp elements $x,\, y :: X$ in Shulman's crisp type theory, but we can quantify over cohesive elements $x,\, y : \flat X$. These amount to the same thing, since if $x$ and $y$ are crisp elements of $X$, then $x^{\flat}{}_{\flat} =y^{\flat}{}_{\flat}$ is the same type as $x = y$.

In Agda, which has incorporated the $\flat$ modality since version 2.6, we can quantify over crisp variables.
  \end{rmk}

That we can think of locally discrete types as being types \emph{of} discrete objects is justified by the following lemma.
\begin{lem}
The type $\Type_{\shape}$ of discrete types is locally discrete.
\end{lem}
\begin{proof}
For any modality, the types of identifications between modal types are equivalent to modal types. In particular, $\Type_{\shape}$ is separated relative to the canonical extension of $\shape$ to any universe containing $\Type$.
\end{proof}

In \cite{LocalizationInHoTT}, Christensen, Opie, Rijke, and Scoccola show that
if a modality ${\modal}$ is given by localization at a type $X$, then the
${\modal}$-separated types also form a modality whose operator is given by
localization at the suspension $\Sigma X$ (see Lemma 2.15 and Remark 2.16 of
\cite{LocalizationInHoTT}). As a corollary, by Lemma \ref{lem:modality.facts} we get that locally discrete types are closed under dependent sums.
\begin{lem}
If $X$ is locally discrete and $P : X \to \Type$ is a family of locally discrete
types, then $\dsum{x : X} P x$ is locally discrete. \qedhere
\end{lem}

We can package this result into a useful extension of the idea that a locally discrete type is a type \emph{of} discrete objects. Many structured objects are captured by the notion of a \emph{standard notion of structure}, which appears in the HoTT Book \cite{HoTTBook} in Section 9.8 as a tool to prove the structure identity principle. A standard notion of structure on a category $\Ca$ is a pair $(P, H)$ where $P : \Ca_0 \to \Type$ assigns to each object of $\Ca$ its type of $(P, H)$-structures (and $H$ gives a notion of homomorphism between such structures). For example, a group is a standard notion of structure on the category of sets by letting $P$ take each set to the set of group structures on it. We can read the previous lemma as saying that discretely structured discrete objects are also discrete, in the following way.
\begin{cor}
Let $\Ca$ be a category whose type of objects $\Ca_0$ is locally discrete type, and $(P, H)$ be a standard notion of structure on $\Ca$ such that for all $x : \Ca_0$, $Px$ is discrete. Then the type of $(P,H)$ structures is locally discrete.
\end{cor}
\begin{proof}
The type of structures is just the dependant sum $\dsum{x : \Ca_0} P x$, which is locally discrete by the above corollary.
\end{proof}

There are two ways to say a crisp type $X :: \Type$ is discrete: either $(-)_{\flat} : \flat X \to X$ is an equivalence or $(-)^{\shape} : X \to \shape X$ is an equivalence. Correspondingly, there are two ways to say that a crisp type is locally discrete, which we have given the names of locally discrete and locally crisply discrete. Though a crisp type which is locally discrete will always be locally crisply discrete, these two notions are likely \emph{not} equivalent in general since the latter only quantifies over \emph{crisp} elements of $X$. We can, however, give another characterization of locally crisply discrete types.
\begin{lem}
A crisp type $X$ is locally crisply discrete if and only if $(-)_{\flat} : \flat X \to X$ is an embedding.
\end{lem}
\begin{proof}
Recall the left exactness of $\flat$ (Theorem 6.1 of \cite{RealCohesion}); we have an equivalence $\flat(x = y) \simeq (x^{\flat} = y^{\flat})$ for all crisp $x,\, y :: X$ making the following diagram commute:
\begin{center}
    \begin{tikzcd}
\flat(x = y) \arrow[rd, "(-)_{\flat}"'] \arrow[rr, "\simeq"] &       & x^{\flat} = y^{\flat} \arrow[ld, "\ap_{(-)_{\flat}}"] \\
                                                             & x = y &
\end{tikzcd}
\end{center}

Now, $X$ is locally crisply discrete if and only if the downwards map on the left is an equivalence, and $(-)_{\flat}$ is an embedding if and only if the downwards map on the right is an equivalence.
\end{proof}

Let's turn our attention to classifying types. In general, any type $X$ can be seen as ``classifying'' the maps into it. This rather abstract way of thinking is more useful the more readily the objects of $X$ can be turned into types, since maps into $\Type$ correspond to arbitrary bundles of types. For an $x : X$, the following general definition gives a classifying type for ``bundles of $x$s".
\begin{defn}
For a type $X$ and a term $x : X$, we define
$$\BAut_X(x) :\equiv \dsum{y : X} \trunc{x = y}$$
\end{defn}

This notation is inspired by the notation for the classifying space $\BB G$ of principal $G$-bundles for a topological group $G$. If $G \simeq \Aut_X(x)$ is the group of automorphisms of some object (as, for example, $\term{GL}_n(\Rb) \simeq \Aut_{\Vect_{\Rb}}(\Rb^n)$), then $\BAut_X(x)$ as defined above does classify principal $G$-bundles. If $\Aut_X(x)$ has a recognizable name $G$, we will write $\BB G$ for $\BAut_X(x)$.

We will now show that if $X$ is crisply locally discrete, and $x :: X$ is a crisp element, then $\BAut_X(x)$ is discrete.
\begin{lem}
For any crisp type $X$ and crisp $x :: X$, we have an equivalence $\flat\BAut_X(x) \simeq \BAut_{\flat X}(x^{\flat})$ making the following triangle commute:
\[\begin{tikzcd}
\flat\BAut_X(x) \arrow[rd, "(-)_{\flat}"'] \arrow[rr, "\simeq"] &            & \BAut_{\flat X}(x^{\flat}) \arrow[ld, "{(y,p) \mapsto y_{\flat},\, \ldots}"] \\
                                                                & \BAut_X(x) &
\end{tikzcd}\]
\end{lem}
\begin{proof}
Consider the following equivalence:
\begin{align*}
    \flat\BAut_X(x) &:\equiv \flat\big(\dsum{y : X} \trunc{x = y}\big) \\
        &\simeq \dsum{u : \flat X} \mbox{$\term{let}$ $y^{\flat}:\equiv u$ $\term{in}$ $\flat \trunc{x = y}$ }\\
        &\simeq \dsum{u : \flat X} \mbox{$\term{let}$ $y^{\flat}:\equiv u$ $\term{in}$ $\trunc{\flat(x = y)}$ }\\
        &\simeq \dsum{u : \flat X} \mbox{$\term{let}$ $y^{\flat}:\equiv u$ $\term{in}$ $\trunc{x^{\flat} = y^{\flat}}$ }\\
        &\simeq \BAut_{\flat X} (x^{\flat}).
\end{align*}
The first equivalence follows from Lemma 6.8, the second from Corollary 6.7, and the third from Theorem 6.1 of \cite{RealCohesion}. The final equivalence follows from Lemma 4.4 of \cite{RealCohesion}, which says that $(\mbox{let $y^{\flat} := u$ in $f(y^{\flat})$}) = f(u)$.

On $(y, p)^{\flat} : \flat\BAut_X(x)$, this equivalence yields $(y^{\flat},\, \cdots) : \BAut_{\flat X}(x^{\flat})$, and so when applying $(-)_{\flat}$ to either side, we find that the result is the same.
\end{proof}

\begin{thm}\label{thm:locally.crisply.discrete.BAut.discrete}
Suppose $X$ is locally crisply discrete and $x :: X$. Then $\BAut_X(x)$ is (crisply) discrete.
\end{thm}
\begin{proof}
By the above lemma, it suffices to prove that $(y,\cdot) \mapsto (y_{\flat},\cdot) : \BAut_{\flat X}(x^{\flat}) \to \BAut_X(x)$ is an equivalence. Now, $(-)_{\flat} : \flat X \to X$ is an embedding because $X$ is locally crisply discrete, so the map in question is an embedding as well. We just need to show it is surjective.

Suppose $y : \BAut_X(x)$. To prove surjectivity, we need to inhabit $\trunc{\fib(y)}$. Because we are trying to prove a proposition, we may assume that $p : x = y$; but then $(x^{\flat}, p) : \fib(y)$.
\end{proof}

\section{Examples of $\shape$-Fibrations}\label{sec:examples}

By using Theorem \ref{thm:locally.crisply.discrete.BAut.discrete}
together with Theorem \ref{thm:fibration.iff.locallyconstant.modal.fibers}, we
get a nice trick for showing that a map $f : X \to Y$ is a $\shape$-fibration. We
just need give a crisply discrete type $F :: \Type_{\shape}$ such that $\shape
\fib_f (y)$ is merely equivalent to $F$ for all $y : Y$.

\begin{thm}\label{thm:characterizing.shape.fibration}
  Let $f : X \to Y$. If there is a crisp type $F :: \Type_{\shape}$ such that for all $y
  : Y$, $\trunc{F = \shape \fib_f(y)}$, then $f$ is a $\shape$-fibration. If
  furthermore we have that $\trunc{F = \fib_f(y)}$ for all $y : Y$, then $f$ is
  $\shape$-{{\'e}tale}. If $F$ is an $n$-type, then $f$ is a $\shape_{n+1}$-fibration
  (resp. $\shape_{n+1}$-{\'e}tale).
\end{thm}
\begin{proof}
By hypothesis, $\shape \fib_f$ factors through $\BAut(F)$. Since $F$ is a crisp
element of a locally discrete type, $\BAut(F)$ is discrete by Theorem
\ref{thm:locally.crisply.discrete.BAut.discrete} and therefore
$\shape \fib_f$ factors through $\shape Y$. But then, by Theorem
\ref{thm:fibration.iff.locallyconstant.modal.fibers}, $f$ is a $\shape$-fibration.
The second claim follows in the same way from Lemma
\ref{lem:etale.iff.locallyconstant.modal}. If $F$ is an $n$-type, then
$\BAut(F)$ is an $(n+1)$-type, and so the maps factor further through
$\shape_{n+1} X$.
\end{proof}

With a little effort, we can extend this trick to classify fibrations over
disconnected spaces whose fibers over each part are different. A little care
must be taken around crispness.
\begin{cor}
  Let $X,\, Y :: \Type$ and $f :: X \to Y$. Assuming the crisp axiom of choice, $f$ is a $\shape$-fibration if and only
  if there is a $F :: \trunc{\shape Y}_0 \to \Type$ such that for all $y : Y$,
  $\trunc{F(|y^{\shape_0}|) = \shape \fib_f(y)}$.
\end{cor}
\begin{proof}
  First, if there is an $F :: \trunc{\shape Y}_0 \to \Type$ such that for all $y :
  Y$, $\trunc{F(|y^{\shape_0}|) = \shape \fib_f(y)}$, then $\shape\fib_f : Y \to
  \Type$ factors through $\dsum{u : \trunc{\shape Y}_0} \BAut(F(u))$. Since
  $\trunc{\shape Y}_0$ is crisply discrete (by Proposition \ref{prop:truncation.shape}) and for all $z : \flat \trunc{\shape Y}_{0}$ we have that $(\mbox{let $v^{\flat} := z$ in $\type{isdiscrete}(\BAut(F(v)))$})$ by Theorem \ref{thm:locally.crisply.discrete.BAut.discrete}, we find that $\dsum{u : \trunc{\shape Y}_0} \BAut(F(u))$ is crisply discrete by Theorem 6.20 of \cite{RealCohesion}. Therefore, $\shape \fib_f$ factors through $(-)^{\shape}$, proving that $f$ is an
  $\shape$-fibration.

  On the other hand, suppose that $f$ is a fibration. Assuming the crisp axiom
  of choice (Theorem 6.30 of \cite{RealCohesion}), there is a crisp section $s ::
  \trunc{\shape Y}_0 \to Y$ of $|(-)^{\shape}|_0 : Y \to \trunc{\shape Y}_0$; that is, we
  may choose an element in every fiber. Define $F(u) :\equiv \shape \fib_f(su)$.
  It remains to show that $\trunc{F(|y^{\shape}|_0) = \shape \fib_f(y)}$ for all $y
  : Y$. Since $f$ is a fibration, we have that $\shape \fib_f = \fib_{\shape f}
  \circ (-)^{\shape}$ and so
$$ \trunc{F(|y^{\shape}|_0) = \shape \fib_f(y)} \simeq \trunc{\fib_{\shape
    f}((s|y^{\shape}|_0)^{\shape}) = \fib_{\shape f}(y^{\shape})}$$
It will suffice to show that $\trunc{s|y^{\shape}|_0^{\shape} = y^{\shape}}$. But this
is equivalent to $|s|y^{\shape}|_0^{\shape}|_0 = |y^{\shape}|_0$, which holds since
$s$ is a section.
\end{proof}

We can now use Theorem \ref{thm:characterizing.shape.fibration} to give a number of examples of $\shape$-fibrations. In this section, we will be working in real cohesion, assuming that $\shape$ is given by localization at the type $\Rb$ of Dedekind real numbers. We will add two more examples later, in Section \ref{}

\subsection{The Universal Cover of the
  Circle}\label{subsec:universal.cover.of.circle}

We will now show that the map $(\cos,\, \sin) : \Rb \to \Sb^1$ is a
$\shape$-fibration, where $\Sb^1$ is the unit circle in $\Rb^2$. In Section \ref{sec:Covering.Theory}, we will show that it is
indeed the universal cover of the circle $\Sb^1$.

\begin{lem}\label{lem:universal.cover.of.circle}
The map $(\cos,\, \sin) : \Rb \to \Sb^1$ is $\shape_1$-{{\'e}tale}, and so in
particular is a $\shape$-fibration.
\end{lem}
\begin{proof}
Let $r \equiv (\cos,\, \sin)$. Over $(x, y) : \Sb^1$, the fiber of $r$ is
$r^{\ast}(x, y) :\equiv \{\theta : \Rb \mid \cos \theta = x,\, \sin \theta = y\}$. We will show that $r^{\ast}(x, y)$ is merely equivalent to $\Zb$.

For any $\theta : r^{\ast}(x, y)$ and $k : \Zb$, we have that $\theta + 2\pi k $
is in $  r^{\ast}(x, y)$. This gives map $\lam{k} \theta + 2\pi k : \Zb \to r^{\ast}(x, y)$. Moreover, given any other $\varphi : r^{\ast}(x, y)$, the difference $\varphi - \theta$ is an integral multiple of $2\pi$, which gives us a map $\lam{\varphi} \frac{\varphi - \theta}{2 \pi} : r^{\ast}(x, y) \to \Zb$. These maps are clearly inverse, and since $r$ is merely surjective there is always some $\theta$ we may choose to make this equivalence.

We have therefore shown that $r^{\ast} : \Sb^1 \to \Type$ factors through
$\BAut(\Zb)$.\footnote{In fact, since the fibers are actually $\Zb$-torsors,
  $r^{\ast}$ factors through $\BB \Zb$, which would work just as well.} But
$\Zb$ is a crisply discrete set, so by Theorem
\ref{thm:characterizing.shape.fibration}, $r$ is a fibration.
\end{proof}

We can now use the fact that $(\cos,\, \sin)$ is a fibration to calculate the
fundamental group of the circle.
\begin{thm}\label{thm:fundamental.group.of.circle}
Let $\Sb^1$ be the unit circle in $\Rb^2$. Then $\Omega \shape \Sb^1 \simeq \Zb$.
\end{thm}
\begin{proof}
Since
$$\Zb \to \Rb \to \Sb^1$$
is a fiber sequence and $(\cos,\, \sin)$ is a $\shape$-fibration,
$$\Zb \to \ast \to \shape\Sb^1$$
is a fiber sequence, showing that $\Omega \shape \Sb^1 \simeq \Zb$.
\end{proof}
\subsection{Hopf Fibrations}\label{subsec:hopf.fibrations}

In the following, let $\Kb$ be the real numbers $\Rb$, the complex numbers
$\Cb$, or the quaternions $\Hb$. We will denote the apartness relation on any of these number systems by $x \# y$; for real numbers this means $|x - y| > 0$, and for the other two number systems this means $\trunc{x - y} > 0$. If $X$ is a set with an apartness relation and $x : X$, we will denote by $X \# \{x\}$ the set of elements $y : X$ with $x \# y$.

\begin{rmk}
In the presence of Shulman's Axiom T of \cite{RealCohesion}, the notions of apartness and non-equality in $\Rb$, $\Cb$, and $\Hb$ coincide (see Theorem 8.32 of that paper). In this case, we could replace all instances of apartness by non-equality. Otherwise, we make no use of Axiom T.
  \end{rmk}

\begin{defn}
  A \emph{line} in $\Kb^{n+1}$ is a proposition $\La : \Kb^{n+1} \to \Prop$ satisfying:
  \begin{enumerate}
  \item There is (merely) an $x \# 0$ element in $\La$ which is apart from $0$.
    \item For any element $x$ in $\La$ and $c : \Kb$, the scaled element
      $c x$ is in $\La$.
    \item For any elements $x$ and $y$ in $\La$, there is a unique $c :
      \Kb$ such that $c x = y$.
  \end{enumerate}
  For a line $\La$, we define $\{\La\} :\equiv \dsum{x : \Kb^{n+1}} \La(x)$ to be
  its extent. We denote the type of lines in $\Kb^{n+1}$ by $\Kb P^n$.
\end{defn}

  Quite obviously, every line is somehow identifiable with $\Kb$.
  \begin{lem}
    Let $\La : \Kb P^n$ be a line. Then
$$\trunc{\{\La\} = \Kb}.$$
  \end{lem}
  \begin{proof}
    Since we are proving a proposition and since there exists a element apart from zero
    on $\La$, we may assume we have such an element $x$. Then the map $y \mapsto
    c$ where $c$ is the unique element of $\Kb$ such that $c x
    = y$ determines a map $\{\La\} \to \Kb$. Since for any $c : \Kb$,
    $c x$ is on $\La$, this map is surjective. It is injective by the
    uniqueness condition (3).
  \end{proof}

  For any $x : \Kb^{n+1} \# \{0\}$, we get the line $\Kb x$ in the
  direction of $x$ defined as
$$\Kb x (y) :\equiv \exists c : \Kb,\, c x = y.$$

 We have a function $\tilde{h} : \Kb^{n+1}\# \{0\} \to \Kb P^n$, sending
  $x$ to $\Kb x$. We refer
  to its restriction $h : \Sb_{\Kb^{n+1}} \to \Kb P^n$ to the unit sphere of
  $\Kb^{n+1}$ as the \emph{generalized Hopf map}.

  Suppose that $\La : \Kb P^n$ is a line and consider the fiber
  $\fib_{\tilde{h}}(\La)$. By definition, this is the type of all elements $x : \Kb^{n+1} -
  \{0\}$ such that $\Kb x = \La$.
  \begin{lem}\label{lem:fiber.of.hopf.fibration}
    For any line $\La : \Kb P^n$,
    $$\fib_{\tilde{h}}(\La) = \{\La\} \# 0$$
    And, as a corollary,
  $$\fib_h(\La) = \dsum{x : \{\La\}} (\trunc{x} = 1)$$
  consists of the elements on the line $\La$ of unit length.
  \end{lem}
  \begin{proof}
    Suppose that $x$ is in $\La$. By property $2$, $c x$ is in $\La$ for
    any $c : \Kb$, and by property $3$, every element of $\La$ may be so
    expressed in a unique way. Therefore, $\Kb x = \La$.

    On the other hand, if $\Kb x = \La$, then in particular $1 \cdot x = x$ is
    in $\La$.
  \end{proof}

  Putting together these two lemmas, we conclude that for all $\La : \Kb P^n$,
  the fiber of $h$ over $\La$ is merely equivalent to the unit sphere of $\Kb$:
  $$\trunc{\fib_h(\La) = \Sb_{\Kb}}.$$
  In particular, their homotopy types are merely equivalent, and so by Theorem
  \ref{thm:characterizing.shape.fibration},
  $$\Sb_{\Kb} \to \Sb_{\Kb^{n+1}} \to \Kb P^n$$
is a $\shape$-fibration.

  Substituting $\Rb$, $\Cb$, and $\Hb$ back in for $\Kb$, we see that:
\begin{thm}
  $\quad$

\begin{itemize}
\item $\Sb^0 \to \Sb^n \to \Rb P^n$ is a $\shape$-fibration.\footnote{We will see in the next
  section that it is a covering map.}
\item $\Sb^1 \to \Sb^{2n + 1} \to \Cb P^n$ is a $\shape$-fibration. This includes the
  original Hopf fibration $\Sb^1 \to \Sb^3 \to \Cb P^1$.
\item $\Sb^3 \to \Sb^{4n + 3} \to \Hb P^n$ is a $\shape$-fibration. This includes the
  quaternionic Hopf fibration $\Sb^3 \to \Sb^7 \to \Hb P^1$.
\end{itemize}
\end{thm}

\subsection{A $\shape$-Fibration which is not a Hurewicz Fibration}

In this example we will prove that the projection of the $x$ and $y$-axes onto
the $x$-axis is a $\shape$-fibration. This is a classic example of a
quasi-fibration which is not a Hurewicz fibration, since the $x$-axis cannot be
lifted to a path going through a point $y \neq 0$ in the fiber over $x = 0$.

First, we need a useful and straightforward lemma.
\begin{lem}\label{lem:strong.contraction.implies.contractible}
  Let $X$ be a type with a point $x_0 : X$ and suppose that for every $x : X$,
  we have a path $\gamma_x : \Rb \to X$ with $\gamma_x(0) = x$ and $\gamma_x(1)
  = x_0$. Then $\shape X$ is
  contractible.
\end{lem}
\begin{proof}
  Define the map $\tilde{\gamma} : \Rb \to (X \to X)$ by $\tilde{\gamma}(t)(x) =
  \gamma_x(t)$ and note that
  $\tilde{\gamma}(0) = \id_X$ and $\tilde{\gamma}(1) = \term{const}_{x_0}$, the
  constant map at $x_0$. This gives us an identification $\id_x^{\shape} =
  \term{const}_{x_0}^{\shape}$ in $\shape(X \to X)$. It remains to show that such an
  identification implies that $\shape X$ is contractible.

  The functorial action of $\shape$ gives a map $(X \to X) \to (\shape X \to \shape
  X)$, and since the latter is $\shape$-modal this factors uniquely through
  $\shape(X \to X)$. By construction, the map $\shape(X \to X) \to (\shape X \to \shape
  X)$ sends $\id_X^{\shape}$ to $\shape \id_X$, which equals $\id_{\shape X}$ by
  functoriality. Furthermore, $\term{const}_{x_0}^{\shape}$ gets sent
to $\shape(\term{ const }_{x_0}) = \shape(x_0 \circ !)$ where $! : X \to \ast$ is the terminal morphism. By functoriality, this equals the
  composite $\shape X \xto{\shape !} \shape \ast \xto{\shape x_0} \shape X$, which is the constant map at $x_0^{\shape}$. Therefore, the
  identity of $\shape X$ factors through a constant map, and so $\shape X$ is
  contractible.
\end{proof}

\begin{rmk}
We can think of the function $\gamma_{(-)}(-) : X \to (\Rb \to X)$ of Lemma \ref{lem:strong.contraction.implies.contractible} as a weak form of multiplicative action of $\Rb$ on $X$. If we write $t \cdot x :\equiv \gamma_{x}(t)$, then the assumptions $\gamma_{x}(0) = x_{0}$ and $\gamma_{x}(1) = x$ read as $0 \cdot x = x_{0}$ and $1 \cdot x = x$. Seen this way, Lemma \ref{lem:strong.contraction.implies.contractible} shows us that any type with such a multiplicative action of $\Rb$ --- say, a vector space --- is $\shape$-connected.
  \end{rmk}

  As a corollary, we find that the projection
  $$\{(x,\, y ) : \Rb^2 \mid xy = 0\} \to \{x : \Rb\}$$
  is $\shape$-connected (and is therefore in particular a $\shape$-fibration).
  The fiber of this projection over $x : \Rb$ is $\{y : Y \mid xy = 0\}$, and
  for every $y$ in the fiber we have the path $t \mapsto ty$ from $0$ to $y$.

  \begin{rmk}
    We shouldn't expect all quasi-fibrations to be $\shape$-fibrations. The
    closest analogue of a quasi-fibration in real hohesion would be a map $f : X \to Y$ such that
    for every \emph{crisp}
    $y :: Y$, $\gamma : \shape\fib_f(y) \to \fib_{\shape f}(y^{\shape})$ is an
    equivalence. This is strictly weaker than our definition of $\shape$-fibration; it
    amounts to the claim that the pullback of $f$ along $(-)_{\flat} : \flat Y
    \to Y$ is a $\shape$-fibration.
  \end{rmk}

\section{Homotopy Quotients are $\shape$-Fibrations.}\label{sec:higher.groups}

In this section, we show that the quotient map $X \to X \sslash G$ from a
type $X$ to the homotopy quotient $X \sslash G$ of $X$ by an action of the
$\infty$-group $G$ is a fibration whenever $G$ is crisp. If the
action is crisp and transitive, then for any crisp point $x :: X$, the map $G \to
X$ given by acting on $x$ is a fibration as well. We will then give two more examples of $\shape$-fibrations.

Before we prove these things, we should review the definition of $\infty$-group
and $\infty$-group action. These notions can be found in \cite{HigherGroups}, which develops
the basic theory of $\infty$-groups and proves a stabilization theorem about
them.

\begin{defn}
An $\infty$-group is a type $G$ identified with the loop space $\Omega \BB G$ of
a pointed, $0$-connected type $\BB G$ (called the \emph{delooping} of $G$).
Since singleton types are contractible, the type of $\infty$-groups is equivalent
to the type of pointed, $0$-connected types.
\begin{align*}
  \infty\mbox{-Grp} :&\equiv \dsum{G : \Type} \dsum{\BB G : \Type_{\ast}^{>0}} ( G = \Omega \BB G ) \\
  &\simeq \Type_{\ast}^{>0}.
\end{align*}
For this reason, we will often identify $G$ with $\Omega \BB G$.
\end{defn}

We may think of the elements of $\BB G$ as $G$-torsors, and the point $\pt_{\BB
  G} : \BB G$ as $G$ acting on itself. Indeed, for any group $G$ in
the axiomatic sense (a set equipped with operations satisfying laws), we may
construct its delooping $\BB G$ as the type of $G$-torsors, pointed at $G$.

\begin{defn}
  An \emph{action} of the $\infty$-group $G$ on types is a map $X^{(-)} : \BB G
  \to \Type$. We write $X :\equiv X^{\pt_{\BB G}}$ for the image of the point
  $\pt_{\BB G} : \BB G$.

  Given an element $g : G$, we get an automorphism of $X$ by
  applying $X^{(-)}$ to $g$. That is, given $x : X$, define
  $$g x :\equiv \ap(X^{(-)},\, g) \at x.\footnote{where $\at : (f = g) \to
    \dprod{x : X} fx = gx$ is the function that applies an equality of functions
  at a point.}$$
\end{defn}

We can think of an action $X^{(-)} : \BB G \to \Type$ as an action of $G$ on
$X :\equiv X^{\pt_{\BB G}}$, and we can think of the image $X^t$ of $t : \BB G$
as the action of $G$ on $X$ twisted by the torsor $t$.

\begin{defn}
Given an action $X^{(-)} : \BB G \to \Type$, and $x,\, y : X$, define
\begin{align*}
    x \xmapsto[G]{} y &:\equiv \dsum{g : G} (g x = y) \\
    \type{Orbit}(x) &:\equiv \dsum{y : X} ( x \xmapsto[G]{} y ) \\
    \type{Stab}(x) &:\equiv x \xmapsto[G]{} x
\end{align*}

We say that the action is \emph{free} if for all $x,\, y : X$, $x \xmapsto[G]{}
y$ is a proposition and \emph{transitive} if $\trunc{x \xmapsto[G]{} y}$.
\end{defn}

With this terminology in hand, we can easily define the homotopy quotient of a
type by the action of an $\infty$-group.

\begin{defn}
If $X^{(-)} : \BB G \to \Type$ is an action of the $\infty$-group $G$, then
$$X \sslash G :\equiv \dsum{t : \BB G} X^t$$
is the \emph{homotopy quotient} of $X$ by $G$. The quotient map $[-] : X \to X
\sslash G$ is defined by
$$[x] :\equiv (\pt_{\BB G}, x).$$
\end{defn}

This definition is justified by the computation of identity types in dependent
pair types.
\begin{lem}\label{lem:identifications.in.homotopy.quotient}
Let $ X^{(-)} : \BB G \to \Type$ be an action of the $\infty$-group $G$ and
$x,\, y : X$. Then
$$([x] = [y]) \simeq ( x \xmapsto[G]{} y )$$
\end{lem}
\begin{proof}
  This follows immediately from Theorem 2.7.2 of \cite{HoTTBook} after expanding the
  definition of each side.
\end{proof}

Following through the definitions, we get the following long fiber sequence
associated to any $\infty$-group action.

\begin{prop}\label{prop:group.action.fiber.sequence}
For any $\infty$-group $G$, action $X^{(-)} : \BB G \to \Type$, and point $x : X^{}{\pt}$, there is a long fiber sequence ending
\begin{center}
    \begin{tikzcd}
\cdots \arrow[r]  & \type{Stab}(x) \arrow[r] & \type{Orbit}(x) \arrow[lld, out=-30, in=150, "\fst"'] \\
X^{\pt} \arrow[r] & X \sslash G \arrow[r]    & \BB G
\end{tikzcd}
\end{center}
In particular, for all $x : X$, $\type{Orbit}(x) \simeq G$.
\end{prop}

Now we can prove our main theorem for this section.
\begin{thm}\label{thm:group.action.fibrations}
Let $G$ be a crisp $\infty$-group, and $X^{(-)} : \BB G \to \Type$ an action of
$G$. Then the quotient map $[-] : X \to X \sslash G$ is a $\shape$-fibration.

If furthermore $X^{(-)}$ is crisp, then the classifying map $\fst : X \sslash G
\to \BB G$ is a $\shape$-fibration, and if the action is transitive and $x :: X$,
then the map $g \mapsto gx : G \to X$ is a $\shape$-fibration.
\end{thm}
\begin{proof}
  Each fact follows quickly from Proposition
  \ref{prop:group.action.fiber.sequence} and Theorem \ref{thm:characterizing.shape.fibration}.

  Since $\BB G$ is $0$-connected, the map $x \mapsto [x] :\equiv (\pt_{\BB G},
  x)$ is surjective. Since by Proposition \ref{prop:group.action.fiber.sequence}
  the fiber $\fib_{[-]}([x]) \simeq G$ for all $x : X$; in particular for all
  $(t, y) : X \sslash G$ we have a term of $\trunc{\fib_{[-]}((t, y)) = G}$.
  Since $G$ is crisp, we may take the homotopy type of each side to discover (by
  Theorem \ref{thm:characterizing.shape.fibration}) that $[-] : X \to X \sslash
  G$ is a $\shape$-fibration.

 If $X^{(-)}$ is crisp, then so is $X :\equiv X^{\pt_{\BB G}}$ (since the
 $\infty$-group $G$, and hence its delooping $\BB G$ and its basepoint $\pt_{\BB
 G}$ are assumed crisp). Since $\BB G$ is $0$-connected, all the fibers of $\fst
: X \sslash G \to \BB G$ are merely equivalent to $X$, and therefore their
homotopy types are merely equivalent to its homotopy type. So, by Theorem
\ref{thm:characterizing.shape.fibration}, the classifying map $\fst : X \sslash
G \to \BB G$ is a $\shape$-fibration.

Suppose that $x :: X$. If the action is transitive, then for any $y : X$,
$\trunc{\type{Stab}(y) = \type{Stab}(x)}$. Since $x$ is crisp, so is
$\type{Stab}(x)$, so by Theorem $\ref{thm:characterizing.shape.fibration}$ this
proves that the map $g \mapsto g x : G \to X$ (whose fiber over $y : X$ is
$\type{Stab}(y)$ by Proposition \ref{prop:group.action.fiber.sequence}) is a $\shape$-fibration.
\end{proof}

We can use Theorem \ref{thm:group.action.fibrations} to give two more examples of $\shape$-fibrations.

\subsection{$\textbf{SO}(n) \to \textbf{SO}(n+1) \to \Sb^n$}\label{subsec:rotation.of.spheres}

We will first construct a delooping $\BB \textbf{SO}(n)$ of the special
orthogonal group, and then define the action of $\textbf{SO}(n+1)$ on the
$n$-sphere as a map $\BB \textbf{SO}(n+1) \to \Type$ (with $n \geq 1$). We will prove that the fiber of the map $\textbf{SO}(n + 1) \to \Sb^n$
given by acting on the base point has fiber $\textbf{SO}(n)$. Finally, by
Theorem \ref{thm:group.action.fibrations}, we will conclude that the map
$\textbf{SO}(n + 1) \to \Sb^n$ is a $\shape$-fibration.

\begin{defn}
An \emph{orientation} on a normed real $n$-dimensional vector space $V$ is a
unit length element of its exterior power $\Lambda^n V$, equipped with the norm
$$\langle v_1 \wedge \cdots \wedge v_n,\, w_1 \wedge \cdots \wedge w_n \rangle
:= \det [\langle v_i,\, w_j \rangle_V]$$

We define $\BB \textbf{SO}(n)$ to be the type of normed real $n$-dimensional vector spaces $V$ equipped
with an orientation that are merely isomorphic to
$\Rb^n$ with its standard norm and orientation. We point $\BB \textbf{SO}(n)$ at
$\Rb^n$ with its standard norm and orientation.
\end{defn}

We need to justify this definition of $\BB \textbf{SO}(n)$.
\begin{lem}
  $\Omega \BB \textbf{SO}(n) = \textbf{SO}(n)$.
\end{lem}
\begin{proof}
A linear automorphism of $\Rb^n$ which preserves the norm is given by an
orthogonal matrix. If this furthermore preserves the standard orientation on
$\Rb$, that means its $n^{\text{th}}$-exterior power is the identity; but this
is given by multiplying by its determinant, so its determinant must be $1$.
\end{proof}

We can now define the action of $\textbf{SO}(n+1)$ on the $n$-sphere $\Sb^n$.
\begin{defn}
For $(V, \langle -,\,- \rangle)$ a normed vector space, let $\Sb_V :\equiv \{v : V \mid \trunc{v} = 1\}$
be its unit sphere. Note that $\Sb_{\Rb^n} \equiv \Sb^{n-1}$ by definition.

The map $(V, \langle -,\, - \rangle, \omega) \mapsto \Sb_V : \BB
\textbf{SO}(n+1) \to \Type$ induces the action of $\textbf{SO}(n+1)$ on $\Sb^n$.
\end{defn}

\begin{lem}
The action of $\textbf{SO}(n+1)$ on $\Sb^n$ is transitive, and the stabilizer of
the basepoint $1 : \Sb^n$ may be identified with $\textbf{SO}(n)$.
\end{lem}
\begin{proof}
  For $v : \Sb^n$, consider $v$ as a unit vector in $\Rb^{n+1}$. Then $v$ may be
  merely extended to a orthonormal basis of $\Rb^{n+1}$ by the Gram-Schmidt
  process. The resulting matrix will have determinant
  either $1$ or $-1$, but since $\{-1, 1\}$ has decidable equality, we can choose
  to swap two of these basis vectors to get a special orthogonal matrix that
  sends $(1, 0, \ldots, 0) : \Sb^n$ to $v$.

  The stabilizer of the basepoint $1 : \Sb^n$ may be identified with the special
  orthogonal matrices whose first column has its first entry $1$ and all other
  entries $0$. Since the matrix is orthogonal, there can be nothing but $0$s in
  the first row as well. Therefore, the bottom minor given by removing the first
  row and first column is also special orthogonal, and this gives an
  identification of the stabilizer with $\textbf{SO}(n)$.
\end{proof}

Finally, by Theorem \ref{thm:group.action.fibrations}, we may conclude that
$$\textbf{SO}(n) \to \textbf{SO}(n+1) \to \Sb^n$$
is a $\shape$-fibration.

\subsection{A $\shape$-fibration over a $1$-type}\label{sec:wedge.fibration.example}

So far we have only seen $\shape$-fibrations over sets. But with Cohesive HoTT, we
can work directly with topological stacks as well. In this example, we will see
an example of a $\shape$-fibration over a $1$-type --- a stacky version of the real numbers.

Often, a map will fail to be a fibration at a few points because it is ramified
there. For example, the map $\Rb \vee \Rb \to \Rb$ induced by the identity maps
\[
\begin{tikzcd}
\ast \arrow[r, "0"] \arrow[d, "0"'] & \Rb \arrow[d] \arrow[ddr, bend left, "\id"] & \\
\Rb \arrow[r] \arrow[drr, bend right, "\id"']& \Rb \vee \Rb \arrow[dr, dashed] & \\
& & \Rb
\end{tikzcd}
\]
is almost a $\shape$-fibration (indeed, almost a covering), but it is ramified over $0$.
However, when such a ``ramified fibration'' appears as the quotient of a group
action, it can be rectified into a $\shape$-fibration by replacing the base by the
homotopy quotient.

In the above example, note that we can also see this map as the quotient
$$\Rb \vee \Rb \to \Rb \vee \Rb / C_2$$
of the action of the cyclic group $C_2$ of order 2 on $\Rb \vee \Rb$ given by
permuting the factors. The homotopy quotient $\Rb \vee \Rb \sslash C_2$ will be
a stacky version of the reals where $0$ has automorphism group $C_2$. Now the
fiber over $0$ consists of both a point over $0$ (of which there is just one),
together with an identification of its image with $0$, of which there are now
two. So the fibers have become locally constant; they are in fact merely
equivalent to the group $C_2$.

This can be made formal by appealing to the upcoming Theorem
\ref{thm:group.action.fibrations}. We will construct the example above.

\begin{defn}
Let $\BB C_2$ be the type of $2$-element sets pointed at $\{0,\,1\}$, noting that $C_2 = \Omega \BB C_2$.

For $T : \BB C_2$, let $X^T$ be the cofiber of $(\id, 0) : T \to T \times
\Rb$. Note that $X :\equiv X^{\pt_{\BB C_2}}$ may be identified with $\Rb \vee
\Rb$. This gives the action of $C_2$ on $\Rb \vee \Rb$ by permuting the factors.
\end{defn}

Theorem \ref{thm:group.action.fibrations} then tells us that
$$C_2 \to \Rb \vee \Rb \to \Rb \vee \Rb \sslash C_2$$
is a $\shape$-fibration. Explicitly $\Rb \vee \Rb \sslash C_2$ is the type of
pairs $\dsum{T : \BB C_2} X^T$ of 2-element sets $T$ and elements of the cofiber
of the inclusion $(\id, 0): T \to T \times \Rb$.

A map can be a ``ramified fibration'' even if each fiber\footnote{That is, over
  each \emph{crisp} point.} is the same. An example of this is the M{\:o}bius
band given by rotating $[-1,1]$ around a circle with a half turn mapping down
onto $[-1,1] / \term{sgn}$ sending each longitudinal circle to the set of points
it intersects in a fixed copy of $[-1,1]$ in the M{\:o}bius band.

Each fiber of this map is a circle, but as one travels from $[1]$ to $[0]$ in
$[-1,1] / \term{sgn}$, the fibers double over. So while each fiber is the same,
they do not have a well defined transport along paths as a $\shape$-fibration
would. The trick here is the word ``each''; it is true that every fiber is a
circle over each \emph{crisp} point of $[-1,1] / \term{sgn}$, but not over a generic
point as Theorem \ref{thm:characterizing.shape.fibration} requires.

This ramification can be fixed by considering the map to $[-1,1] \sslash
\term{sgn}$, a stacky version of $[0,1]$ in which $0$ has an automorphism group
$C_2$.

\section{The Shape of a Crisp $n$-Connected Type is $n$-Connected}\label{sec:connectedness}

One might expect that if $X$ is $\trunc{-}_n$-connected, then its homotopy type
$\shape X$ would also be $\trunc{-}_n$-connected. While we do not know whether
this is true in general, we can prove it for crisp types $X :: \Type$. To do
this, we need to recall a bit of the theory of \emph{separated} types for a
modality from \cite{LocalizationInHoTT}.

\begin{defn}
  A type $X$ is ${\modal}$-separated if for all $x,\, y : X$, the type of
  identifications $x = y$ is ${\modal}$-modal. By Theorem 2.26 of \cite{LocalizationInHoTT}, the ${\modal}$-separated types
  form a modality ${\modal}^{\prime}$, and we may inductively define
  \begin{align*}
    {\modal}^{(0)} &:\equiv {\modal} \\
    {\modal}^{(n+1)} &:\equiv {\modal}^{(n)\prime}
  \end{align*}
\end{defn}

We now need to import a few lemmas from \cite{LocalizationInHoTT}.
\begin{lem}\label{lem:separated.modality.units}
Any ${\modal}$-modal type is ${\modal}^{(n)}$-modal, and the canonical factorization
${\modal}^{(n)} X \to {\modal} X$ of the ${\modal}$-unit through the ${\modal}^{(n)}$-unit
is a ${\modal}$-unit.
\end{lem}
\begin{proof}
By hypothesis, the identification types in ${\modal} X$ are ${\modal}$-modal, so
that ${\modal} X$ is ${\modal}'$-modal, and so on. The proves the first statement.

The second statement now follows by Lemma \ref{lem:submodality.connecting.unit}.
\end{proof}

\begin{lem}\label{lem:derived.modality.loop}
  For any modality ${\modal}$ and any pointed type $X$, there is an equivalence
  $$ \Omega^n{\modal}^{(n)} X \simeq {\modal} \Omega^n X$$
\end{lem}
\begin{proof}
This follows immediately from Proposition 2.27 of \cite{LocalizationInHoTT} by induction.
\end{proof}

\begin{lem}
Suppose that ${\modal}$ is given by localization at a map $A \to \ast$. Then
${\modal}^{(n)}$ is given by localization at $\Sigma^n A \to \ast$.
\end{lem}
\begin{proof}
This follows immediately from Lemma 2.15 of \cite{LocalizationInHoTT} by induction.
\end{proof}

As a corollary, we find that the $n$-fold locally discrete modalities
$\shape^{(n)}$ are given by localization at $\Sigma^n \Rb \to \ast$. Since $\Rb$
is inhabited, as a corollary we find that $\shape^{(n)}$ preserves $n$-connected types.
\begin{lem}\label{lem:X.n.connected.esh.n.X.n.connected}
Suppose that $-1 \leq k \leq n$. If $X$ is $k$-connected, then $\shape^{(n)} X$ is $k$-connected.
\end{lem}
\begin{proof}
This follows immedately from Corollary 3.13 of \cite{LocalizationInHoTT} by
induction. In particular, since $\Rb$ is $(-1)$-connected, by Theorem 8.2.1 of
\cite{HoTTBook} $\Sigma^n \Rb$ is $(n-1)$-connected and so $(k-1)$-connected. Corollary 3.13 of \cite{LocalizationInHoTT} then applies
to the map $\Sigma^n \Rb \to \ast$.
\end{proof}

We are now ready to prove that $\shape$ preserves $n$-connected crisp types.
\begin{thm}\label{thm:shape.preserves.crisp.connectedness}
Let $X :: \Type$ be a crisp, $n$-connected type for $n \geq -1$. Then the canonical map
$\shape^{(n + 1)} X \to \shape X$ induced by factoring the $\shape$-unit through the
$\shape^{(n + 1)}$-unit is an equivalence, and so in particular $\shape X$ is $n$-connected.
\end{thm}
\begin{proof}
For $n \equiv -1$, the statement follows tautologically. It remains to show that
assuming the statement for $n$ implies $n + 1$. We note here that since $\Nb$ is
crisply discrete, we may assume all natural numbers are crisp.

First, we argue that we may assume that $X$ is crisply pointed. Since $X$ is
$( n + 1 )$-connected and $n \geq -1$, in particular $\trunc{X}$ is contractible and so
also $\flat \trunc{X}$ is contractible. By
Corollary 6.7 of \cite{RealCohesion}, $\flat \trunc{X} \simeq \trunc{\flat X}$
so that $\trunc{\flat X}$ is also contractible. Since we are trying to prove
that a map is an equivalence, which is a proposition, we may assume that we have
a $u : \flat X$, and therefore assume that we have $u \equiv x^{\flat}$ for a
crisp $x :: X$.

Now, assume that $x :: X$ is a crisp point of $X$ and that $X$ is $(n +
1)$-connected. Then $\Omega X$ is a crisp, $n$-connected type and therefore
$\shape^{(n+1)}\Omega X \to \shape \Omega X$ is an equivalence by hypothesis; in
partiuclar $\shape^{(n+1)} \Omega X$ is discrete. Therefore,
$\shape^{(n+1)}\Omega X \simeq \Omega \shape^{(n+2)}X$ is discrete. By Lemma
\ref{lem:X.n.connected.esh.n.X.n.connected}, $\shape^{(n+2)} X$ is $(n +
1)$-connected and therefore in particular $0$-connected; therefore, it is
locally crisply discrete. Since it is pointed and $0$-connected, it is also equivalent to
$\BAut_{\shape^{(n+2)}X}(x^{\shape^{(n+1)}})$ and
so by Theorem \ref{thm:locally.crisply.discrete.BAut.discrete}, it is discrete.
But then the canonical map $\shape^{(n+2)} X \to \shape X$ is an equivalence by
Lemma \ref{lem:separated.modality.units}.
\end{proof}

Using Theorem \ref{thm:shape.preserves.crisp.connectedness}, we can show that
the homotopy type of a higher group is a higher group.
\begin{defn}
  A \emph{$k$-commutative $\infty$-group} is a type $G$ identified with
  $\Omega^{k+1} \BB^{k+1} G$ for a pointed, $k$-connected type $\BB^{k+1}
  G$.\footnote{In \cite{HigherGroups}, $k$-commutative $\infty$-groups are
    called $(k+1)$-tuply groupal, but I couldn't bear to subject the reader to
    such terminology.}
  A homomorphism of $k$-commutative $\infty$-groups is a pointed map $\BB^{k+1}
  G \to \BB^{k + 1} H$.
\end{defn}

\begin{lem}\label{lem:derived.modality.loop.naturality}
  The equivalence ${\modal} \Omega^{(n)} = \Omega^{(n)} {\modal}^{(n)}$ of Lemma
  \ref{lem:derived.modality.loop} is natural. Let $f : X \pto Y$ be a pointed map between pointed
  types. Then the following square commutes:
  \[
    \begin{tikzcd}
      {\modal} \Omega^{n} X \arrow[r, "{\modal} \Omega^{n} f"] \arrow[d, "\sim"'] & {\modal} \Omega^n Y \arrow[d,
      "\sim"] \\
      \Omega^n {\modal}^{(n)} X \arrow[r, "\Omega {\modal}^{(n)} f"'] & \Omega^n
      {\modal}^{(n)} Y
    \end{tikzcd}
  \]
\end{lem}
\begin{proof}
  Since $\Omega^n {\modal}^{(n)} Y$ is modal, we may check that this commutes on $p
  : \Omega^n X$. When restricted to $\Omega^n X$, the square becomes $\Omega^n$
  applied to the ${\modal}^{(n)}$-naturality square, which commutes.
\end{proof}

\begin{thm}\label{thm:shape.of.infty.group.is.infty.group}
 Suppose that $G$ is a crisp, $k$-commutative $\infty$-group with $(k+1)$-fold
 delooping $\BB^{k+1}G$. Then $\shape G$ is a $k$-commutative $\infty$-group with
 delooping $\shape \BB^{k+1} G$ and the unit $(-)^{\shape} : G \to \shape G$ is a homomorphism.
\end{thm}
\begin{proof}
  By Theorem \ref{thm:shape.preserves.crisp.connectedness}, $\shape \BB^{k+1} G$
  is $k$-connected and may be pointed at $\pt_{\BB^{k+1} G}^{\shape}$. By the same
  theorem,
  \begin{align*}
    \Omega^{k+1} \shape \BB^{k+1} G &\simeq \Omega^{k+1} \shape^{(k+1)} \BB^{k+1} G \\
    &\simeq \shape \Omega^{k+1} \BB^{k+1} G \\
    &\simeq \shape G.
  \end{align*}
  By Lemma \ref{lem:derived.modality.loop.naturality} and the fact that the
  composite $\BB^{k+1} G \to \shape^{(k+1)} \BB^{k+1} G \xto{\sim} \shape \BB^{k+1}
  G$ is equal to the unit $\BB^{k+1} G \to \shape \BB^{k+1} G$, this unit deloops
  the unit $G \to \shape G$, showing that the latter is a $k$-commutative homomorphism.
\end{proof}

As a corollary, we can understand the homotopy type of some classifying types.
\begin{itemize}
\item Let $\BB \type{GL}_1(\Rb)$ be the type of $1$-dimensional real vector
  spaces. Since $\shape \type{GL}_1(\Rb) = \{-1,\, 1\}$ may be identified with the
  group of signs, we get find that $\shape \BB \type{GL}_1(\Rb) = \BB \Zb/2$. We
  can call the $\shape$-unit $w_1 : \BB \type{GL}_1(\Rb) \to \BB \Zb/2$ the
  \emph{first Stiefel-Whitney class}, since pushing forward by it sends a real
  line bundle to a first degree cocycle in $\Zb/2$ cohomology. Since this is a $\shape$-unit, we see that the first Stiefel-Whitney class is the universal discrete cohomological invariant of a real line bundle.
\item Let $\BB \type{U}(1)$ be the type of $1$-dimensional normed complex vector
  spaces. Since $\shape \type{U}(1) = \BB \Zb$ is a pointed, connected type whose
  loop space is $\Zb$, we find that $\shape \BB \type{U}(1) = \BB^2 \Zb$. We can
  call the $\shape$-unit $c_1 : \BB \type{U}(1) \to B^2 \Zb$ the \emph{first Chern
  class}, since pushing forward by it sends a Hermitian line bundle to a second
degree cocycle in integral cohomology.
Since this is a $\shape$-unit, we see that the first Chern class is the universal discrete cohomological invariant of a complex line bundle.
\end{itemize}

We can now show, with a quick modal argument, that the first Chern class of the
Hopf fibration generates $H^2(\Sb^2;\Zb)$.
\begin{prop}
The first Chern class $c_1(h)$ of the Hopf fibration $h : \Sb^3 \to \Sb^2$
generates $H^2(\Sb^2; \Zb)$.
\end{prop}
\begin{proof}
  For the purpose of this proof, we make an identification of $\Sb^2$ with
  $\Cb P^1$ and so take the points of $\Sb^2$ to be complex lines in $\Cb^2$. We
  will show that the $\shape_2$-unit $\Sb^2 \to \shape_2 \Sb^2$
  generates $H^2(\Sb^2; \Zb)$, and then that $c_1(h)$ factors uniquely through
  this unit.

  Consider the long exact sequence of homotopy groups associated to the Hopf
  fibration. Since we have calculated (in Lemma
  \ref{lem:universal.cover.of.circle}) that $\Omega \shape \Sb^1 \simeq \Zb$, we
  see that $\pi_2(\shape\Sb^2) \simeq \pi_1(\shape\Sb^1) = \Zb$. Therefore, $\shape_2 \Sb^2$
  is a $\BB^2 \Zb$, and the $\shape_2$-unit $(-)^{\shape_2} : \Sb^2 \to \shape_2
  \Sb^2$ induces the identity on $\pi_2$ and so generates $H^2(\Sb^2; \Zb)$.

  It remains to show that $c_1(h) : \Sb^2 \to \BB^2 \Zb$ is an $\shape_2$-unit. Let $\chi : \Sb^2 \to \BB \type{U}(1)$ send a line $\La : \Sb^2$ in
  $\Cb^2$ to $\{\La\}$,
  the normed $1$-dimensional complex vector space that it is as a subspace of $\Cb^2$. This
  classifies the Hopf fibration by Lemma \ref{lem:fiber.of.hopf.fibration} and because a unitary isomorphism with $\Cb$ is
  determined by an element of unit norm:
  $$\fib_{\chi}(\Cb) \equiv \dsum{\La :
    \Sb^2} (\{\La\} = \Cb) \simeq \dsum{\La : \Sb^2} \dsum{\ell :
    \{\La\}}(\trunc{\ell} = 1) \simeq \dsum{\La : \Sb^2} \fib_{h}(\La)$$
  In other words, $c_1(h) \equiv c_1 \circ \chi$. Now, the fibers of $\chi$ are
  merely equivalent to $\Sb^3$, and $\shape_2 \Sb^3 = \ast$, so it is
  $\shape_2$-connected. But $c_1$ is an $\shape_2$-unit and so also
  $\shape_2$-connected. Therefore, $c_1 \circ \chi$ is a $\shape_2$-connected map
  into a $\shape_2$-modal type; by Lemma 1.38 of \cite{RSS}, it is therefore a $\shape_2$-unit.

\end{proof}

\section{A Bit of Covering Space Theory}\label{sec:Covering.Theory}

In this section, we'll see a bit of modal covering theory and get a sense of how
working with coverings using modalities feels. In his \emph{Cohesive Covering
  Theory} extended abstract \cite{WellenCovering}, Wellen defines a
\emph{modal covering} map $\pi : E \to B$ for a modality ${\modal}$ to be a
${\modal}$-{\'e}tale map. He then specializes to the modality $\shape_1$ to recover
the usual covering theory. Here, in light of further conversation with Wellen,
we will make a slightly less general definition of covering map which relates
more closely to the traditional theory.

\begin{defn}
  A map $\pi : E \to B$ is a \emph{cover} if it is $\shape_1$-{\'e}tale and its
  fibers are sets.
\end{defn}

Recall from Section \ref{sec:modality.refresher} that ${\modal}$-equivalences lift
uniquely against ${\modal}$-{\'e}tale maps. In particular, in any square
\[
  \begin{tikzcd}
    \ast \arrow[r] \arrow[d,"0"' ] & E \arrow[d,"\pi" ] \\
    \Rb \arrow[r] & B
  \end{tikzcd}
\]
there is a unique filler since $\Rb$ is $\shape_1$-connected.
Therefore, covers satisfy the \emph{unique} path lifting property.

We can quickly prove the classical theorem that coverings of a space $X$
correspond to actions of the fundamental groupoid of $X$ on discrete sets.
\begin{thm}
  Let $X$ be a type and let $\mbox{Cov}(X)$ denote the type of covers of $X$. Then
  $$\mbox{Cov}(X) \simeq (\shape_1 X \to \Type_{\shape_0}).$$
\end{thm}
\begin{proof}
  This follows immediately from Corollary \ref{cor:characterizing.etale.maps},
  applied to the modality $\shape_1$. This corollary says that $\shape_1$-{\'e}tale
  maps into $X$ correspond to maps from $\shape_1 X$ to $\Type_{\shape_1}$. If furthermore
  the fibers are sets, then the maps go from $\shape_1 X$ to $\Type_{\shape_0}$.
\end{proof}

 Classically, the universal cover is just any simply connected cover. We can let
 this characterization lead us to a definition of the universal cover of a
 pointed, homotopically connected space. Let $X$ be a space and $\pi : \tilde{X} \to X$ a covering with
 $\tilde{X}$ simply connected in the sense that $\shape_1 \tilde{X} = \ast$. Since
 $\pi$ is a covering, and hence $\shape_1$-{{\'e}tale}, the $\shape_1$-naturality
 square
 \[
 \begin{tikzcd}
   \tilde{X} \arrow[d,"\pi"'] \arrow[r] & \shape_1\tilde{X} \arrow[d,"\shape_1 \pi"]
   \\
   X \arrow[r] & \shape_1 X
 \end{tikzcd}
 \]
 is a pullback. But $\shape_1 X = \ast$, so this shows us that $\tilde{X} =
 \fib_{(-)^{\shape_1}}(u)$ for some $u : \shape_1 X$. This leads us to the following
 definition.
 \begin{defn}
   Let $X$ be a type and $\pt_X : X$ a base point. Suppose further that $X$ is
   homotopically connected in the sense that $\trunc{\shape_1 X}_{0} = \ast$. Then
   the \emph{universal cover} $\pi : \tilde{X} \pto X$ is defined to be $\fst :
   \fib_{(-)^{\shape_1}}(\pt_X^{\shape_1}) \to X$, with $\pt_{\tilde{X}} :\equiv
   (\pt_X,\, \refl)$ and $\pt_{\pi} :\equiv \refl$:
   \[
 \begin{tikzcd}
   \tilde{X} \arrow[d,"\pi"'] \arrow[r] & \ast \arrow[d,"\pt_X^{\shape_1}"]
   \\
   X \arrow[r] & \shape_1 X
 \end{tikzcd}
   \]
 \end{defn}

 \begin{thm}
The universal cover $\pi : \tilde{X} \to X$ is the initial pointed cover of $X$.
That is, for any pointed cover $c : C \pto X$, there is a unique pointed cover $\chi_c
  : \tilde{X} \pto C$ such that $c \pcirc \chi_c = \pi$ as pointed maps.
 \end{thm}
 \begin{proof}
   We need to show that the universal cover is a cover with the correct
   universal property.

   First, note that as the fiber of a $\shape_1$-unit, $\tilde{X}$ is
   $\shape_1$-connected (that is, simply connected).
   Therefore, the naturality square
 \[
 \begin{tikzcd}
   \tilde{X} \arrow[d,"\pi"'] \arrow[r] & \shape_1\tilde{X} \arrow[d,"\shape_1 \pi"]
   \\
   X \arrow[r] & \shape_1 X
 \end{tikzcd}
 \]
 is equal to the square
 \[
 \begin{tikzcd}
   \tilde{X} \arrow[d,"\pi"'] \arrow[r] & \ast \arrow[d,"\pt_X^{\shape_1}"]
   \\
   X \arrow[r] & \shape_1 X
 \end{tikzcd}
\]
which is a pullback. As the $\shape_1$-naturality square of $\pi$ is a pullback,
$\pi$ is $\shape_1$-{\'e}tale. The fiber of $\pi$ over any point $x : X$ is
equivalent to $x^{\shape_1} = \pt_X^{\shape_1}$, which is a type of identifications
in the $1$-type $\shape_1 X$ and is therefore a set. This proves that $\pi$ is a cover.

Now for the universal property. Note that since $\pi(\pt_{\tilde{X}}) \equiv \pt_X$, the data of a pointed cover
$c : C \pto X$ can be expressed as a square
\[
  \begin{tikzcd}
    \ast \arrow[r,"\pt_C" ] \arrow[d,"\pt_{\tilde{X}}"' ] & C \arrow[d,"c" ] \\
    \tilde{X} \arrow[r,"\pi" ] & X
  \end{tikzcd}
\]
in which the map $c$ is a cover. A filler of that square is precisely a pointed
map $\tilde{X} \to C$ over $X$. But $\tilde{X}$ is $\shape_1$-connected and
therefore the map $\pt_{\tilde{X}} : \ast \to \tilde{X}$ is an
$\shape_1$-equivalence. And since $c$ is a $\shape_1$-{\'e}tale map and
$\shape_1$-equivalences are orthogonal to $\shape_1$-{\'e}tale maps by Lemma 6.1.23
of \cite{RijkeThesis}, the type of fillers of this square is contractible.

It remains to show that the unique filler of the square is a cover. Since $c$
and $\pi$ are $\shape_1$-{\'e}tale, it is $\shape_1$-{\'e}tale. And since $c$ and
$\pi$ have set fibers, it does as well. Therefore, it is a cover.
 \end{proof}

 As promised, Lemma \ref{lem:universal.cover.of.circle} does prove that
 $(\cos,\, \sin) : \Rb \to \Sb^1$ is the universal cover of the circle. This map is
 $\shape_1$-{\'e}tale, its fibers are sets, and $\Rb$ is simply connected.

Theorem \ref{thm:characterizing.shape.fibration} provides us with a simple trick
for showing that a map is a cover.
\begin{cor}\label{cor:characterizing.covers}
 Let $\pi : E \to B$. If there is a crisply discrete set $F$ such that
 $\trunc{\fib_{\pi}(b) = F}$ for all $b : B$, then $\pi$ is a cover.
\end{cor}

\begin{rmk}
As promised in Section \ref{subsec:hopf.fibrations}, the map $\Sb^{n+1} \to \Rb
P^n$ is a covering map, and since $\Sb^{n+1}$ is simply connected for $n \geq 0$,
this is the universal cover of $\Rb P^n$.
\end{rmk}

We can prove a seemingly suspect proposition with this trick: any map with
finite fibers is a cover. To do this, we need to prove a bit of folklore.
\begin{lem}\label{lem:finite.types}
Let $\type{Fin} :\equiv \dsum{X : \Type} \trunc{\dsum{n : \Nb} { X = \{1,\ldots,
  n\} }}$ be the type of finite types (types $X$ for which there exists an $n$ such that $X =
\{1,\ldots,n\}$). There is an equivalence
$$\type{Fin} \simeq \dsum{n :
  \Nb} \BAut(n)$$
between the type of finite types and the sum over $n : \Nb$ of the classifying types $\BAut(n)
:\equiv \dsum{X : \Type} \trunc{X = \{1,\ldots,n\}}$ of the symmetric group
$\Aut(n)$.
\end{lem}
\begin{proof}
Note that
\begin{align*}
  \dsum{n : \Nb} \BAut(n) &\equiv \dsum{n : \Nb} \dsum{X : \Type} \trunc{X = \{1, \ldots, n\}} \\
  &\simeq \dsum{X : \Type} \dsum{n : \Nb} \trunc{X = \{1, \ldots, n\}}.
\end{align*}
Therefore, it will suffice to show that $\dsum{n : \Nb} \trunc{X = \{1,\ldots,
  n\}} \simeq \trunc{\dsum{n : \Nb} {X = \{1, \ldots, n\}}}$ assuming that $X :
\Type$. But the obvious map $(n, |p|) \mapsto |(n, p)|$ is a $\trunc{-}$-unit by
Lemma 1.24 of \cite{RSS}, so it will suffice to show that $\dsum{n : \Nb}
\trunc{X = \{1,\ldots, n\}}$ is a proposition.

Suppose that $(n, p)$ and $(m, q)$ are of type $\dsum{n : \Nb} \trunc{X = \{1,
  \ldots, n\}}$, seeking $(n, p) = (m, q)$. From $p$ and $q$, we get
$\trunc{\{1, \ldots, n\} = \{1, \ldots, m\}}$. A simple induction shows that
this occurs if and only if $n = m$.
\end{proof}

\begin{prop}
  Let $\pi : E \to B$ be a map whose fibers are finite in the sense that for
  every $b : B$, there exists an $n : \Nb$ such that $\trunc{\fib_{\pi}(b) =
    \{1,\ldots,n\}}$. Then $\pi$ is a cover.
\end{prop}
\begin{proof}
Note that this condition says that the map $\fib_{\pi} : B \to \Type$ factors
through $\type{Fin} \hookrightarrow \Type$. But by Lemma \ref{lem:finite.types},
$\type{Fin} \simeq \dsum{n : \Nb} \BAut(n)$, and since $\Nb$ is crisply
discrete, we have an equivalence
$$\dsum{n : \Nb} \BAut(n) \simeq \dsum{n : \flat \Nb} \mbox{let $n := m^{\flat}$ in
$\BAut(m)$}.$$
Now, in the inner expression, $m :: \Nb$ is \emph{crisp}, and so Theorem
\ref{thm:locally.crisply.discrete.BAut.discrete} applies and $\BAut(m)$ is
discrete. Therefore, $\type{Fin}$ is a discretely indexed sum of discrete types,
and so it is also discrete. It is, futhermore, a $1$-type since it is a set
indexed sum of $1$-types.

Therefore, $\fib_b$ factors through $\shape_1 B$ and so by Lemma
\ref{lem:etale.iff.locallyconstant.modal}, is $\shape_1$-{\'e}tale. By hypothesis,
its fibers are finite and therefore sets, so it is a cover.
\end{proof}

\begin{rmk}
What is strange about this theorem is that there appear to be counterexamples.
Consider the map $\Rb \vee \Rb \to \Rb$ we looked at in Example
\ref{sec:wedge.fibration.example}. It seems like its fibers are finite. By a
quick application of descent, we can see that its fiber over $r : \Rb$ is
equivalent to the suspension $\Sigma (r = 0)$ of the proposition that $r = 0$.
The inclusion of the endpoints of the suspension are always jointly surjective,
so there is a \emph{surjection} $\{0,1\} \to \Sigma(r = 0)$. But we cannot prove
this is a bijection, or that there is a bijection from $\Sigma(r = 0)$ to
$\{0\}$ without deciding the proposition $r = 0$. We can't decide whether
a real number is $0$ (since the reals are connected), so we can't find a precise
cardinality for the fiber. This example emphasizes the difference between
cardinal finiteness (being equivalent to some $\{1, \ldots, n\}$) and
Kuratowski finiteness (admitting a surjection from some $\{1, \ldots, n\}$) in
real cohesion.
\end{rmk}

\begin{rmk}
While the map $\Rb \vee \Rb \to \Rb$ we considered in Example
\ref{sec:wedge.fibration.example} is not a covering, the homotopy quotient $\Rb \vee \Rb \to
\Rb \vee \Rb \sslash C_2$ \emph{is} a cover, and is in fact the universal cover
of $\Rb \vee \Rb \sslash C_2$. To see this, note that $\Rb \vee \Rb$ is
contractible since it is given as a crisp pushout and $\shape$ preserves crisp
pushouts. The fibers of the homotopy quotient are merely equivalent to $C_2$,
which is a discrete set, so the map is a covering. This gives an example of the
universal cover of a space which is not a set.
\end{rmk}

For a particular example of these results, consider an $n$-fold cover of the
circle $\Sb^1$.
\begin{defn}
  An $n$-fold cover $\pi : E \to B$ is a map whose fibers have $n$ elements. By Corollary \ref{cor:characterizing.covers}, an $n$-fold cover is indeed a cover.
\end{defn}

\begin{thm}\label{thm:n.fold.covers.cirlce}
Let $n : \Nb$. The type of $n$-fold covers of $\Sb^1$ whose fiber over $(1, 0)$ is identified with a fixed $n$-element set $\{1,\ldots,n\}$ is equivalent to the type $\Aut(n)$ of permutations of $n$ elements.
\end{thm}
\begin{proof}
First, we note that since $\Nb$ is crisply discrete, we may assume without loss of generality that $n$ is crisp and that the fixed $n$-element set $\{1,\ldots,n\}$ is also crisp. The type in question is
$$\dsum{f : \Sb^1 \to \BAut(n)}{( f(1,0) = \{1,\ldots,n\} )}$$
the type of \emph{pointed} maps from the circle to $\BAut(n)$. But Theorem
\ref{thm:locally.crisply.discrete.BAut.discrete}, $\BAut(n)$ is discrete and so this is equivalent to the type
$$\dsum{f : \shape \Sb^1 \to \BAut(n)} ( f(1,0)^{\shape} = \{1,\ldots,n\} ).$$
By Theorem 9.5 of \cite{RealCohesion}, $(\Sb^1 \to X) \simeq (S^1 \to X)$ for any discrete $X$, and so the above type is equivalent to $$\dsum{f : S^1 \to \BAut(n)} ( f(\pt) = \{1,\ldots,n\} )$$
which, by the universal proposty of $S^1$, is equivalent to $\Omega \BAut(n) \simeq \Aut(n)$.
\end{proof}

\begin{figure}[h]
    \centering
    \begin{tikzpicture}[declare function={f(\x)=0.2*sin(\x)+\x/1000;},
 rubout/.style={/utils/exec=\tikzset{rubout/.cd,#1},
 decoration={show path construction,
      curveto code={
       \draw [white,line width=\pgfkeysvalueof{/tikz/rubout/line width}+2*\pgfkeysvalueof{/tikz/rubout/halo}]
        (\tikzinputsegmentfirst) .. controls
        (\tikzinputsegmentsupporta) and (\tikzinputsegmentsupportb)  ..(\tikzinputsegmentlast);
       \draw [line width=\pgfkeysvalueof{/tikz/rubout/line width},shorten <=-0.1pt,shorten >=-0.1pt] (\tikzinputsegmentfirst) .. controls
        (\tikzinputsegmentsupporta) and (\tikzinputsegmentsupportb) ..(\tikzinputsegmentlast);
      }}},rubout/.cd,line width/.initial=2pt,halo/.initial=0.5pt]
 \draw[rubout={line width=2pt,halo=0.5pt},decorate]
   plot[variable=\x,domain=-50:970,samples=55,smooth] ({cos(\x)},{f(\x)}) to[out=0,in=195] cycle;
 \draw[rubout={line width=2pt,halo=0.5pt},decorate]
   plot[variable=\x,domain=1400:2040,samples=55,smooth] ({cos(\x)},{f(\x)}) to[out=0,in=195] cycle;
 \draw[line width=2pt] (0,-2) arc(-90:270:1cm and 0.2cm);
 \draw[thick,-stealth]  (0,-0.4) -- (0,-1.4);
\end{tikzpicture}
    \caption{A $5$-fold cover of the circle corresponding to the permutation $(12)(354)$. It has cycle type $(2,3)$, corresponding to the $2$ elements of the fiber in the top connected component, and the $3$ elements in the bottom.}
    \label{fig:covering.cycle.type}
\end{figure}
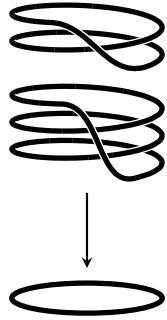

Looking at some examples of $n$-fold coverings (such as Figure \ref{fig:covering.cycle.type}), we might get the idea that the set of connected components of the total space corresponds to the cycle type of its induced permutation. Somewhat more objectively, we might expect that the set of connected components of the total space should correspond to the set of orbits of the action of the induced permutation on the elements of a fiber. We can prove this using a nice modal argument.

\begin{thm}\label{thm:fundamental.groupoid.of.total.space.of.cover}
  Let $\pi : E \to B$ be a cover over a pointed base $B$ with fiber $F$ which is connected in
  the sense that $\shape_1 B$ is $0$-connected. Then
  $$\shape_1 E = F \sslash \pi_1(B)$$
  where $\pi_1(B) :\equiv \Omega(\shape_1 B, \pt_B^{\shape_1})$ is the fundamental
  group of $B$.
\end{thm}
\begin{proof}
Since $\pi : E \to B$ is a cover, $\fib_{\pi} : B \to \Type$ factors through
$\shape_1 B$ as $\fib_{\shape_1 \pi}$:
\[
\begin{tikzcd}
B \arrow[d, "(-)^{\shape_1}"'] \arrow[r, "\fib_{\pi}"] & \Type \\
\shape_1 B \arrow[ur, dashed, "\fib_{\shape_1 \pi}"'] &
\end{tikzcd}
\]
witnessed by $\delta : \fib_{\pi}(b) \xto{\sim} \fib_{\shape_1 \pi}(b^{\shape_1})$. Taking total spaces, we find that the following square is a pullback:
\[
\begin{tikzcd}
E \arrow[d, "\pi"'] \arrow[r, "\term{tot}(\delta)"] & \dsum{t : \shape_1 B} \fib_{\shape_1
  \pi}(t) \arrow[d, "\fst"] \\
B \arrow[r, "(-)^{\shape_1}"'] & \shape_1 B
\end{tikzcd}
\]
Since $(-)^{\shape_1} : B \to \shape_1 B$ is $\shape_1$-connected (by Theorem 1.32 of
\cite{RSS}) and
$\shape_1$-connected maps are preserved under pullback (by Theorem 1.34 of \cite{RSS}), the top map
$\term{tot}(\delta)$ is also $\shape_1$-connected.

 Now, since $\shape_1 B$ is $0$-connected, when pointed at $\pt_B^{\shape_1}$ it can
 be considered as the delooping $\BB \pi_1(B)$ of the fundamental group of $B$. Then, the homotopy quotient $\fib_{\pi}(\pt_B) \sslash \pi_1(B)$ can be constructed
 as the pair type
$$ F \sslash \pi_1(B) :\equiv \dsum{t : \shape_1 B}
\fib_{\shape_1 \pi}(t).$$
See Section \ref{sec:higher.groups} for a brief introduction to the theory of
higher groups and Lemma \ref{lem:identifications.in.homotopy.quotient} for a
justification of this construction.

So, the canonical map $E \to
F \sslash \pi_1(B)$ is $\shape_1$-connected and therefore in
particular a $\shape_1$-equivalence. But as a $\shape_1$-modally indexed sum
of $\shape_1$-modal types, $\fib_{\pi}(\pt_B) \sslash \pi_1(B)$ is $\shape_1$-modal,
so we find that $\shape_1 E = F
\sslash \pi_1(B)$.
\end{proof}

\begin{cor}
Let $\pi : E \to \Sb^1$  be an $n$-fold covering of the circle whose fiber over $(1,0)$ is identified with $\{1,\ldots,n\}$, and let $\varphi : \Aut(n)$ be the corresponding permutation. Then the set of connected components of the total space $E$ is equivalent to the set of orbits of the action of $\varphi$ on $\{1,\ldots,n\}$.
\end{cor}
\begin{proof}
  The set of connected components of the total space may be constructed as
  $\trunc{\shape_1 E}_0$, which by Theorem
  \ref{thm:fundamental.groupoid.of.total.space.of.cover} is equivalent to
  $\trunc{\fib_\pi((1,0)) \sslash \pi_1(\Sb^1)}_0$. As we calculated in Theorem
  \ref{thm:fundamental.group.of.circle}, $\pi_1(\Sb^1) = \Zb$, and by hypothesis
  $\fib_{\pi}((1, 0)) = \{1, \ldots, n\}$. So the connected components of $E$ is
  equivalent to $\trunc{\{1, \ldots, n\} \sslash \Zb}_0$ with the action given by
  $1 \mapsto \varphi$. By Lemma
  \ref{lem:identifications.in.homotopy.quotient}, two elements of $\trunc{\{1,
    \ldots, n\} \sslash \Zb}_0$ are equal if and only if there is an integer
  that sends one to the other; in other words, this is the set of orbits of the
  action of $\varphi$, as desired.
\end{proof}

We can extend the definition of a cover naturally to an ``$n$-cover'' using the
modality $\shape_n$.
\begin{defn}
  A map $\pi : E \to B$ is an $n$-cover if it is $\shape_n$-{\'e}tale and its
  fibers are $(n-1)$-types.
\end{defn}

The theory of $n$-covers follows just as smoothly as the theory of
covers. For every fact above about covers, there is an analogous fact about
$n$-covers proved in the same way. In particular, a universal $n$-cover is just a $\shape_n$-connected $n$-cover. We
can describe the universal $2$-cover of the $2$-sphere.
\begin{thm}
  Let $h : \Sb^3 \to \Sb^2$ be the Hopf fibration. Then the $\shape$-modal factor $\fst : \dsum{s
    : \Sb^2} \shape \fib_h(s) \to \Sb^2$ of the Hopf fibration is the universal $2$-cover of the $2$-sphere.
\end{thm}
\begin{proof}
  Let $\pi : E \to \Sb^2$ denote the $\shape$-modal factor of the Hopf fibration.
  Note that $\fib_{\pi}(s) = \shape \fib_h(s)$ is merely equivalent to the crisply
  discrete 1-type $\shape \Sb^1$ for all $s :
  \Sb^2$, and is therefore by Theorem \ref{thm:characterizing.shape.fibration}
  is $\shape_2$-{\'e}tale and so a $2$-cover. Furthermore, $\shape E \simeq \shape
  \Sb^3$, so it is $\shape_2$-connected (since $\shape \Sb^3 = S^3$ is $2$-connected), and therefore the universal $2$-cover.
\end{proof}

The theory of $n$-covers seems related to the theory of Whitehead towers, but the precise
relationship between these notions in Cohesive HoTT is not yet clear to the author.

We can show that the universal cover of a crisp $\infty$-group is also an
$\infty$-group. If $G$ is a crisp $\infty$-group, then so is $\shape_1 G \simeq
\trunc{\shape G}_1$ by Theorem \ref{thm:shape.of.infty.group.is.infty.group} and
so we get a long fiber sequence:
\begin{center}
    \begin{tikzcd}
  & \cdots \arrow[r] & \pi_1(G) \arrow[lld, out=-30, in=150] \\
  \tilde{G} \arrow[r] & G \arrow[r]    & \shape_1 G  \arrow[lld, out=-30, in=150] \\
  \BB \tilde{G} \arrow[r] & \BB G \arrow[r] & \shape_2 \BB G
\end{tikzcd}
\end{center}
The delooping of $\tilde{G}$ is defined to be the fiber of $(-)^{\shape_2} : \BB
G \to \shape_2 \BB G$, and it is $0$-connected since the unit $(-)^{\shape_1} : G
\to \shape_1 G$ is surjective. Note that $\BB \tilde{G}$ is the universal
$2$-cover of $\BB G$.

We can continue this fiber sequence on as long as
$G$ can be delooped, taking $\shape_{k+1} \BB^k G$ as the delooping of $\shape_k
\BB^{k-1} G$ and taking $\BB^k \tilde{G}$ to be the universal $(k+1)$-cover of
$\BB^k G$. In particular, we get a long fiber sequence:
\begin{center}
    \begin{tikzcd}
  & \cdots \arrow[r] & \Zb \arrow[lld, out=-30, in=150] \\
  \Rb \arrow[r] & \type{U}(1) \arrow[r]    & \BB \Zb  \arrow[lld, out=-30, in=150] \\
  \BB \Rb \arrow[r] & \BB \type{U}(1) \arrow[r] & \BB^2 \Zb \arrow[lld, out=-30,
  in=150]\\
  \cdots & &
\end{tikzcd}
\end{center}
This gives us a long exact sequence $H^{\ast}(-;\,\Zb) \to H^{\ast}(-;\,\Rb) \to
H^{\ast}(-;\,\type{U}(1)) \to H^{\ast + 1}(-;\,\Zb) $ in \emph{continuous} cohomology.

In this paper, we have defined a notion of \emph{modal fibration} and explored
the fibrations for the \emph{shape} modality of Real Cohesive HoTT. We have seen
that it is often quite easy to prove a map is a $\shape$-fibration --- indeed, if
you know what the fiber is ahead of time, it is often trivial. After a fibration
is found, many simple calculations can be done with purely modal arguments.

Though we only briefly discussed them in this paper, the author hopes that
this framework can make calculations in the theory of orbifolds and Lie groupoids
more approachable and more conceptual.

\printbibliography

\end{document}